\documentclass{amsart}

\usepackage[usenames,dvipsnames]{color}

\usepackage{amssymb, latexsym}
\usepackage{float}
\usepackage{graphicx}
\usepackage[shortlabels]{enumitem}

\usepackage[pagewise]{lineno}

\usepackage[shortlabels]{enumitem}

  \newtheorem{corollary}{Corollary}
\newtheorem{lemma}{Lemma}

\newtheorem{theorem}{Theorem}
\newtheorem{proposition}{Proposition}
\newtheorem{claim}{Claim}

\newtheorem{othertheorem}{Theorem}

\theoremstyle{definition}
\newtheorem{definition}{Definition}
\theoremstyle{remark}
\newtheorem{remark}{Remark}

\newtheorem{example}{Example}

\numberwithin{equation}{section}
\numberwithin{theorem}{section}
\numberwithin{definition}{section}
\numberwithin{remark}{section}
\numberwithin{example}{section}
\numberwithin{lemma}{section}
\numberwithin{property}{section}
\numberwithin{proposition}{section}
\numberwithin{claim}{section}
\numberwithin{conj}{section}
\numberwithin{corollary}{section}

\newcommand{\D}{{\mathbb D}}

\newcommand{\C}{{\mathbb C}}
\newcommand{\N}{{\mathbb N}}

\newcommand{\F}{{\mathcal F}}
\newcommand{\J}{{\mathcal J}}

\newcommand{\CC}{{\overline \C}}

\title[Uniformly perfect and HNUP non-autonomous attractors]{Uniformly perfect and Hereditarily non Uniformly Perfect analytic and conformal non-autonomous attractor sets}

\author{Mark Comerford, Kurt Falk, Rich Stankewitz, and Hiroki Sumi}

\thanks{2010 Mathematics Subject Classification: Primary 30D05, 28A80, 37F10.  \\
Key words and phrases:  Attractor sets, Limit sets, Uniformly perfect, Iterated function systems, Non-autonomous iteration,
 Julia sets, Hereditarily non uniformly perfect.
}

\email{mcomerford@math.uri.edu}
\email{falk@math.uni-kiel.de}
\email{rstankewitz@bsu.edu}
\email{sumi@math.h.kyoto-u.ac.jp}

\begin{document}
\maketitle

\centerline{\scshape Mark Comerford}
\medskip
{\footnotesize
 \centerline{Department of Mathematics}
    \centerline{University of Rhode Island}
    \centerline{5 Lippitt Road, Room 102F}
   \centerline{Kingston, RI 02881, USA}
} 

\medskip


\centerline{\scshape Kurt Falk}
\medskip
{\footnotesize
 \centerline{Mathematisches Seminar}
   \centerline{Christian-Albrechts-Universit\"at zu Kiel}
   \centerline{Ludewig-Meyn-Str. 4, 24118 Kiel, Germany}
} 

\medskip

\centerline{\scshape Rich Stankewitz }
\medskip
{\footnotesize
 \centerline{Department of Mathematical Sciences}
   \centerline{Ball State University}
   \centerline{Muncie, IN 47306, USA}
} 

\medskip

\centerline{\scshape Hiroki Sumi }
\medskip
{\footnotesize
 \centerline{Course of Mathematical Science}
 \centerline{Department of Human Coexistence}
 \centerline{Graduate School of Human and Environmental Studies}
    \centerline{Kyoto University}
    \centerline{Yoshida-nihonmatsu-cho, Sakyo-ku}
   \centerline{Kyoto 606-8501, Japan}
} 


\begin{abstract}
Conditions are given which imply that certain non-autonomous analytic iterated function
systems (NIFS's) in the complex plane $\C$ have uniformly perfect
attractor sets, while other conditions imply the attractor is pointwise thin, and thus hereditarily non uniformly perfect.  Examples are given to illustrate the main theorems, as well as to indicate how they generalize other results.  Examples are also given to illustrate how possible generalizations of corresponding results for autonomous IFS's do not hold in general in this more flexible setting.  Further, applications to non-autonomous Julia sets are given.

Lastly, since our definition of NIFS is in some ways more general than others found in the literature, a careful analysis is given to show when certain familiar relationships still hold, along with detailed examples showing when other relationships do not hold.
\end{abstract}

\section{Introduction and statements of the main theorems}\label{Intro}

The aim of this paper is two-fold, the first is a thickness result while the second relates to a corresponding notion of thinness.  In particular, we present conditions that imply the attractors in $\C$ of certain non-autonomous iterated function systems are uniformly perfect, and then, looking to the other extreme, give conditions for attractors to be pointwise thin (and thus hereditarily non uniformly perfect).

Uniformly perfect sets, which are defined in Section~\ref{SectDefsAndFacts}, were
introduced by A.~F.~Beardon and Ch.~Pommerenke in 1978 in~\cite{BP}.
Such sets cannot be separated by annuli that are too large in modulus (equivalently, large
ratio of outer to inner radius).
Thus, uniform perfectness, in a sense, measures how “thick”
a set is near each of its points and is related in spirit to many other notions of thickness such as
Hausdorff content and dimension, logarithmic capacity and density, H\"older regularity, and positive
injectivity radius for Riemann surfaces. For an excellent survey of uniform perfectness and how it
relates to these and other such notions see Pommerenke~\cite{P} and Sugawa~\cite{Sug2}.

The concept of hereditarily non uniformly perfect was introduced in
\cite{SSS} and can be thought of as a thinness criterion for sets which is a strong version of failing to be uniformly perfect.
In particular, a compact set $E \subset \C$ is called \textit{hereditarily non uniformly perfect} (HNUP) if no subset of $E$ is uniformly perfect.  Often a compact set is shown to be HNUP by showing it satisfies the stronger property of \emph{pointwise thinness} (see Definitioin~\ref{DefPointThin}).  This is done in several examples in~\cite{SSS,CSS}, and will be done in this paper each time a set is shown to be HNUP.

When the maps are all complex analytic and the IFS is \textit{autonomous} (see Section~\ref{SecAutAttractMainThm}), uniform perfectness results of the type we seek are found in~\cite{RS4}.  We also note that~\cite{UrbBook} includes related results for similar systems (which require an open set condition).  Certain constructions in~\cite{SSS} are \textit{non-autonomous} iterated function systems shown to have uniformly perfect attractors (though those examples were not presented as attractors, but rather as Cantor-like constructions - see Example~\ref{ExHNUPcantor} in this paper), while other examples there are not uniformly perfect.  We look to generalize those results here, and we begin by following~\cite{RU} to introduce the main framework and definitions (with some key differences) of \emph{non-autonomous iterated function systems} (NIFS's).  We also note that attractors of NIFS's are often Moran-set constructions (see~\cite{Wen} for a good exposition of such).

\begin{definition}[NIFS] \label{defNIFS}
Let $(U, X)$ be a pair where $U$ is a non-empty open connected subset of $\C$ and $X \subset U$ is compact.
A \emph{non-autonomous iterated function system} (NIFS) $\Phi$ on the pair $(U,X)$ is given by a sequence $\Phi^{(1)},\Phi^{(2)},\Phi^{(3)}, \dots$, where each $\Phi^{(j)}$ is a collection of non-constant functions $(\varphi_i^{(j)}:U \to X)_{i \in I^{(j)}}$ such that there exists
$0<s<1$ and a metric $d$ on $U$ where $d(\varphi(z),\varphi(w))\leq s d(z,w)$
for all $z, w \in X$ and all $\varphi \in \cup_{j=1}^\infty \Phi^{(j)}$.  We also stipulate that $d$ induces the Euclidean topology on $X$.  Thus this system is uniformly
contracting on the forward invariant (see definition below) metric space $(X,d)$.
\end{definition}

\begin{definition}[Forward Invariant] \label{defForInv}
A set $\widetilde{X} \subseteq U$ is called \textit{forward invariant under} $\Phi$ when $\varphi(\widetilde{X})\subseteq \widetilde{X}$ for all $\varphi \in \cup_{j=1}^\infty \Phi^{(j)}$.
\end{definition}

We define a NIFS and its corresponding attractor set (see Definition~\ref{LimitSetdef})
to be \emph{analytic} (respectively, \emph{conformal}) if all the maps are complex analytic
(respectively, conformal) on $U$.  Note that here and throughout conformal means analytic and one-to-one (globally on $U$, not just locally).

Important differences from~\cite{RU} in the above setup are: 1)~We do not impose that $X$ have other geometric properties such as convexity or a smooth boundary. 2)~The maps do not need to be conformal.  In fact, they do not even need to be locally conformal.  3)~In~\cite{RU}, the focus is on certain measures and dimension of the attractor sets, and so it is required that each $I^{(j)}$ be a finite or countably infinite index set.  We, however, do not make any such assumption.  4)~We do not in general impose an \textit{open set condition}, and, in fact, there can be substantial overlap in sets of the form $\varphi_a^{(j)}(X)$ and $\varphi_b^{(j)}(X)$.  However, for several of our HNUP results we shall require the \textit{Strong Separation Condition}:
$\varphi^{(j)}_{a}(X) \cap \varphi^{(j)}_{b}(X)=\emptyset,$ for each $j\in \N$ and distinct ${a}, {b} \in I^{(j)}$. 5)~The main object of interest to this paper is the \textit{analytic} NIFS, and so the condition imposed that each $\varphi$ map $U$ into $X$ allows us, under this condition of analyticity, to
take the metric $d$ to be the hyperbolic metric on $U$ (see Section~\ref{SectDefsAndFacts}).

Given an NIFS, we wish to study the limit set (or attractor) which we can define after the next definition.

\begin{definition}[Words] \label{Wordsdef}
For each $k \in \N$, we define the symbolic spaces \[I^k:=\prod_{j=1}^k I^{(j)} \qquad \textrm{and} \qquad  I^\infty:=\prod_{j=1}^\infty I^{(j)}.\]
\end{definition}
Note that a $k$-tuple $(\omega_1, \dots, \omega_k) \in I^k$ may be identified with the corresponding word $\omega_1\dots \omega_k$.  When $\omega^* \in I^\infty$ has $\omega^*_j = \omega_j$ for $j=1, \dots, k$, we call $\omega^*$ an \textit{extension} of $\omega=\omega_1\dots \omega_k \in I^k$.

\begin{definition} \label{LimitSetdef}
For all $k \in \N$ and $\omega= \omega_1 \cdots \omega_k \in I^k$, we define $\varphi_\omega:=\varphi_{\omega_1}^{(1)} \circ \cdots \circ \varphi_{\omega_k}^{(k)}$ with \[X_\omega:= \varphi_\omega (X) \qquad \textrm{      and       } \qquad X_k:= \bigcup_{\omega \in I^k} X_\omega.\]  The \emph{limit set} (or \emph{attractor}) of $\Phi$ is defined as \[J=J(\Phi):=\bigcap_{k=1}^\infty X_k.\]
\end{definition}

\begin{remark}
The attractor $J$ does not have to be compact.  For example, $J$ is not compact for the autonomous system (see Section~\ref{SecAutAttractMainThm}) given in Example 4.3 of~\cite{RS4}.  However, if each index set $I^{(j)}$ is finite, then each $X_k$ is compact and hence so is $J$.
\end{remark}

In order to state the main results, Theorems~\ref{Thm2Main} and~\ref{ThmAnalNIFS} (regarding uniform perfectness) and Theorems~\ref{Thm2MainHNUP} and~\ref{ThmAnalNIFS-HNUP} and Corollary~\ref{CorHNUP} (regarding pointwise thinness), we first present the following notation.

Given an NIFS $\Phi^{(1)},\Phi^{(2)},\Phi^{(3)}, \dots$ on some $(U,X)$, we note that by excluding $\Phi^{(1)},\Phi^{(2)},\dots,\Phi^{(j-1)},$ the sequence $\Phi^{(j)},\Phi^{(j+1)},\Phi^{(j+2)}, \dots$ also forms an NIFS (which formally would be  $\widetilde{\Phi}^{(1)},\widetilde{\Phi}^{(2)},\widetilde{\Phi}^{(3)}, \dots$ where each $\widetilde{\Phi}^{(k)}=\Phi^{(k+j-1)}$).  The new NIFS would then induce sets as in Definition~\ref{LimitSetdef}, which we denote as $X_\omega^{(j)}, X_k^{(j)}$, and $J^{(j)}$ with the superscript used to indicate the relationship to the original NIFS.  In particular, for the original NIFS the sets $X_k$ may also be denoted $X_k^{(1)}$.  See Example~\ref{ExCantor}, illustrated in Figure~\ref{TablePic}, noting that the superscript indicates the column and the subscript indicates the row where a given set resides (noting that row 0 refers to the top row).

\begin{theorem}\label{Thm2Main}
Let $\Phi$ be a conformal NIFS on $(U, X)$.  Suppose
\begin{enumerate}[(i)]
\item (M\"obius Condition)  each map in $\varphi \in \cup_{j \in \N} \Phi^{(j)}$ is M\"obius, and
\item (Two Point Separation Condition) there exists $\delta>0$ such that each $\Phi^{(j)}$, for $j\in \N$, contains  (not necessarily distinct) maps $\varphi^{(j)}_{a}$ and $\varphi^{(j)}_{b}$ such that for some (not necessarily distinct) $z_a, z_b \in J^{(j+1)}$ we have $|\varphi^{(j)}_{a}(z_a)-\varphi^{(j)}_{b}(z_b)|\geq \delta$, and
\item (Derivative Condition) there exists $\eta>0$ such that for all $\varphi \in \cup_{j \in \N} \Phi^{(j)}$ we have $|\varphi'| \geq \eta$ on $X$.
\end{enumerate}
Then each $\overline{J^{(j)}}$ is uniformly perfect.  Furthermore, for a given $(U, X)$, the modulus of any annulus separating any $\overline{J^{(j)}}$ is bounded above by a constant depending only on $\delta$ and $\eta$.
\end{theorem}

\begin{remark}\label{RemSep}
Instead of verifying the Two Point Separation Condition as stated, it is often easier to check any of the increasingly stronger conditions:
\begin{enumerate}[(1)]
\item there exists $\delta>0$ such that each $\Phi^{(j)}$, for $j\in \N$, contains at least two maps $\varphi^{(j)}_{a}$ and $\varphi^{(j)}_{b}$ such that for some $z \in J^{(j+1)}$ we have $|\varphi^{(j)}_{a}(z)-\varphi^{(j)}_{b}(z)|\geq \delta$,
\item there exists $\delta>0$ such that each $\Phi^{(j)}$, for $j\in \N$, contains at least two maps $\varphi^{(j)}_{a}$ and $\varphi^{(j)}_{b}$ such that for \underline{all $z \in X$} we have $|\varphi^{(j)}_{a}(z)-\varphi^{(j)}_{b}(z)|\geq \delta$,
\item \label{ItemSepCond} there exists $\delta>0$ such that each $\Phi^{(j)}$, for $j\in \N$, contains at least two maps $\varphi^{(j)}_{a}$ and $\varphi^{(j)}_{b}$ such that the images $\varphi^{(j)}_{a}(X)$ and $\varphi^{(j)}_{b}(X)$ are at least a distance $\delta$ apart.
\end{enumerate}
Note that~\ref{ItemSepCond} is much weaker than what in the literature is often called the \textit{Strong Separation Condition} for \textit{finite autonomous systems} (stronger than the Strong Separation Condition stated earlier), which can be equivalently stated as such: there exists $\delta>0$ such that for \emph{all distinct maps}  $\varphi^{(j)}_{a}, \varphi^{(j)}_{b} \in \Phi^{(j)}$, for $j\in \N$, the images $\varphi^{(j)}_{a}(X)$ and $\varphi^{(j)}_{b}(X)$ are at least a distance $\delta$ apart.

We also note that this Two Point Separation Condition shows that, for each $j \in \N$, $\mathrm{diam}(J^{(j)})\geq \delta$ since for any $z_a, z_b \in J^{(j+1)}$ and $\varphi^{(j)}_{a}, \varphi^{(j)}_{b} \in \Phi^{(j)}$,  we have, by the inclusion proved in Remark~\ref{RemInvarCondLimitSets}, $\varphi^{(j)}_{a}(z_a), \varphi^{(j)}_{b}(z_b) \in J^{(j)}$.  In the proof of Theorem~\ref{Thm2Main}, the Two Point Separation Condition is only used to obtain a uniform lower bound on $\mathrm{diam}(J^{(j)})$.
\end{remark}

\begin{theorem}\label{ThmAnalNIFS}
Suppose $\Phi$ is an analytic NIFS such that $\overline{J^{(n)}}$, for some integer $n >1$, is uniformly perfect (e.g., when the NIFS given by $\Phi^{(n)},\Phi^{(n+1)}, \Phi^{(n+2)}, \dots,$ satisfies the hypotheses of Theorem~\ref{Thm2Main}).  Suppose also that $\widetilde{\Phi}^{(1)}=\Phi^{(1)} \circ \cdots \circ \Phi^{(n-1)}$ is finite.  Then $\overline{J(\Phi)}$ is uniformly perfect.
\end{theorem}

We now present the main results regarding pointwise thinness, which is defined in Section~\ref{SectDefsAndFacts} along with other relevant terms.  Theorem~\ref{Thm2MainHNUP} and Corollary~\ref{CorHNUP} concern conformal NIFS's having the Strong Separation Condition, and Theorem~\ref{ThmAnalNIFS-HNUP} concerns analytic NIFS's which do not require the Strong Separation Condition but do require a certain type of separation condition.  In order to present these results precisely, we must first introduce the projection map.

\begin{remark}[Projection Map]\label{RemProjMap}
Consider $\omega^* \in I^\infty$ and note that the compact sets $\varphi_{\omega^*_1 \cdots \omega^*_n}(X)$ decrease with   $\mathrm{diam}_d(\varphi_{\omega^*_1 \cdots \omega^*_n}(X)) \leq s^n \mathrm{diam}_d(X) \to 0$ as $n \to \infty$.  Hence $\cap_{n=1}^\infty \varphi_{\omega^*_1 \cdots \omega^*_n}(X)$ contains just a single point that we call $\pi(\omega^*)$.  Note that $\pi(\omega^*) \in J$ since it clearly belongs to each $\varphi_{\omega^*_1 \cdots \omega^*_n}(X) \subseteq X_n$.  We call $\pi_\Phi:I^\infty \to J$ the \textit{projection map}.

Further note that for any non-empty compact $\widetilde{X} \subseteq X$ that is forward invariant under $\Phi$, we have that  $\cap_{n=1}^\infty \varphi_{\omega^*_1 \cdots \omega^*_n}(\widetilde{X})=\cap_{n=1}^\infty \varphi_{\omega^*_1 \cdots \omega^*_n}(X)$ since each is a singleton set with the left set being a subset of the right set.  We summarize this by saying that the projection map $\pi_\Phi$ is independent of the choice of non-empty compact forward invariant set $X$.
\end{remark}

\begin{theorem}\label{Thm2MainHNUP}
Let $\Phi$ be a conformal NIFS on $(U, X)$, with $X$ connected, satisfying the Strong Separation Condition and the following
\begin{description}
  \item[Separating Annuli Condition] there exists a sequence of conformal annuli\\ $\{A_{j_n}\}_{n \in \N}$, where each $A_{j_n}$ and the bounded component of $\C \setminus A_{j_n}$ are in $X$, such that for all $n \in \N$ the annulus $A_{j_n}$ separates $X_1^{(j_n)}$ where $\mathrm{mod}~A_{j_n} \to \infty$ as $n \to \infty$.
\end{description}
For each $n \in \N$, choose $m_n \in I^{(j_n)}$ such that the set $\varphi_{m_n}^{(j_n)}(X)$ is surrounded by $A_{j_n}$ (which can be done since $X$ is connected and $A_{j_n}$ separates $X_1^{(j_n)}$), and fix $\omega = (\omega_1, \omega_2, \dots) \in I^\infty$ such that $\omega_{j_n}=m_n$ for all $n \in \N$.  Then, $J$ is pointwise thin at $\pi_\Phi(\omega)$.
\end{theorem}

\begin{remark}
The Separating Annuli Condition can be visualized in Figure~\ref{TablePic} in Example~\ref{ExCantor} by considering annuli $A_j$ of maximum modulus separating the two components in each $X_1^{(j)}$ (in row 1), noting that $\mathrm{mod}~A_{j} \to \infty$ exactly when $a_j \to 0$.
\end{remark}

\begin{corollary}\label{CorHNUP}
Let $\Phi$ be a conformal NIFS on $(U, X)$, with $X$ connected, satisfying the Strong Separation Condition.  Suppose along some subsequence $j_n$, we have $2 \leq \#\Phi^{(j_n)} < \infty$ for all $n \in \N$.
Define, for each $n \in \N$,
$$b_{j_n}=\min\{\mathrm{dist}(\varphi^{(j_n)}_{i}(X),\partial X):i \in I^{(j_n)} \},$$
$$\delta_{j_n}=\min\{\mathrm{dist}(\varphi^{(j_n)}_{a}(X),\varphi^{(j_n)}_{b}(X)):a,b \in I^{(j_n)} \textrm{ with }a \neq b \}$$  and  $$\eta_{j_n}=\max \{\mathrm{diam}(\varphi_i^{(j_n)}(X)):i \in I^{(j_n)}\}.$$
Suppose for some $c>1$, we have $\delta_{j_n} \leq c b_{j_n}$ for all $n \in \N$.
Further suppose $\frac{\delta_{j_n}}{\eta_{j_n}} \to \infty$ as $n \to \infty$.
Then, $J=\pi_\Phi(I^\infty)$ is pointwise thin (and thus HNUP when $J$ is compact).
\end{corollary}

\begin{remark}
Since each $\delta_{j_n} \leq \mathrm{diam}(X)$, we see that we may choose $c= \frac{\mathrm{diam}(X)}{\inf_n\{b_{j_n}\}}$ when $\inf_n\{b_{j_n}\}>0$.
\end{remark}

\begin{remark}
Since each $\delta_{j_n} \leq \mathrm{diam}(X)$, we see that for $\frac{\delta_{j_n}}{\eta_{j_n}} \to \infty$ we must have $\eta_{j_n} \to 0$.  In such a situation then, $\Phi$ cannot satisfy the Derivative Condition, a critical assumption in the proof of the uniform perfectness of $J$ in Theorem~\ref{Thm2Main} (see Remark~\ref{RemDerivCritical}).
\end{remark}


\begin{remark}
Corollary~\ref{CorHNUP} applies much more generally when we recall that one can combine stages in the manner described in Remark~\ref{RemCombStages}.  Specifically, we may show  ${J(\Phi)}$ is pointwise thin by applying Corollary~\ref{CorHNUP} to any $\widetilde{\Phi}$ created by combining stages in $\Phi$.  This technique of combining stages is used later to analyze Example~\ref{ExHNUP}.
\end{remark}

\begin{theorem}\label{ThmAnalNIFS-HNUP}
Suppose
\begin{enumerate}[(i)]
\item  $\Phi$ is an analytic NIFS such that $\overline{J^{(n)}}$, for some integer $n >1$, is pointwise thin (e.g., when the NIFS given by $\Phi^{(n)},\Phi^{(n+1)}, \Phi^{(n+2)}, \dots,$ satisfies the hypotheses of Corollary~\ref{CorHNUP} with each $\Phi^{(j)}$ finite), and
\item $\widetilde{\Phi}^{(1)}=\Phi^{(1)} \circ \cdots \circ \Phi^{(n-1)}$ is finite with $\varphi_a(\overline{J^{(n)}}) \cap \varphi_b(\overline{J^{(n)}})=\emptyset$ for all distinct $\varphi_a, \varphi_b \in \widetilde{\Phi}^{(1)}$ (e.g., when $\Phi$ satisfies the Strong Separation Condition), and
\item for every $\varphi_a\in \widetilde{\Phi}^{(1)}$ and $z \in \overline{J^{(n)}}$, we have that $z$ is the only point of $\overline{J^{(n)}}$ which maps to $\varphi_a(z)$ (e.g., when each map in $\widetilde{\Phi}^{(1)}$ is conformal).
\end{enumerate}
Then $\overline{J(\Phi)}$ is pointwise thin.
\end{theorem}

The remainder of the paper is organized as follows. Section~\ref{SectPrelim} establishes the preliminary results needed later as well as provides several examples, including Example~\ref{ExCantor} which graphicly highlights key relationships.  This section also identifies some important aspects that show how the systems of study in this paper can be more delicate than related systems found in the literature.
Section~\ref{SecAutAttractMainThm} reviews known results for autonomous attractors (where $I^{(j)}$ and $\Phi^{(j)}$ are independent of $j$) and relates them to the main results for non-autonomous attractors stated in Theorems~\ref{Thm2Main} and~\ref{ThmAnalNIFS}.
Section~\ref{SectDefsAndFacts} contains basic results and definitions.  Section~\ref{SecExamples} presents some examples to demonstrate why the possible generalizations of results for autonomous systems (presented as Theorems~\ref{cor3}-\ref{cor} in Section~\ref{SecAutAttractMainThm}) do not hold for general NIFS's.
In Section~\ref{SecExamples} we show that our main results generalize Theorem 4.1 of~\cite{SSS}.
Section~\ref{SectApplicationsJulia} contains applications of Corollary~\ref{CorHNUP} to non-autonomous Julia sets along polynomial sequences.
Section~\ref{SectProofs} is then used to prove Theorems~\ref{Thm2Main} and~\ref{ThmAnalNIFS} on uniform perfectness, and Theorem~\ref{Thm2MainHNUP}, Corollary~\ref{CorHNUP}, and Theorem~\ref{ThmAnalNIFS-HNUP} on pointwise thinness.

\section{Key preliminaries regarding the projection map, dependence on $X$, invariance conditions, and stage combination}\label{SectPrelim}

We begin this section by establishing some notation.

\noindent \textbf{Notation to be used throughout:}
Let $q$ be a metric.
For a set $F\subseteq
\C$, we define its \textit{diameter} to be $\mathrm{diam}_q F=\sup\{q(z,w):z, w \in F\}$ and $\epsilon$-\textit{ball about} $F$ to be
$B_q(F,\epsilon) = \{z:\mathrm{dist}_q(z,F) <\epsilon\}$ where
$\mathrm{dist}_q(z,F)=\inf\{q(z,w):w \in F\}$.  Also, for $w \in \C$ and $r>0$ we define the \textit{disk} and \textit{circle}, respectively, by
$\Delta_q(w,r)=\{z:q(z,w)<r\}$ and $C_q(w,r)=\{z:q(z,w)=r\}.$ If no
metric is noted, then it is assumed that the metric is the
Euclidean metric.   Lastly, the open unit disk in $\C$ is denoted $\D$.

%

\begin{remark}[Pieces of $X_k$]\label{RemPiecesContainLimitPts}
The limit set $J = \cap_{k=1}^\infty X_k$ is a decreasing intersection of the $X_k$, but an important facet of each $X_k$ is that it is the union of what we call the \textit{pieces} of $X_k$, each of which must contain both a limit point and a fixed point.  More precisely, note that for any $k \in \N$ and $\omega= \omega_1 \cdots \omega_k \in I^k$, we have that the \textit{piece} $\varphi_\omega(X)$ of $X_k$, for which $\mathrm{diam}_d(\varphi_\omega(X))\leq s^k\mathrm{diam}_d(X)$, contains both the fixed point of the contraction $\varphi_\omega$ and the point $\pi_\Phi(\omega^*) \in J$ for any extension $\omega^* \in I^\infty$ of $\omega$.  Note also that the pieces of $X_k$ are not necessarily components of $X_k$ since the pieces may overlap in general.
\end{remark}

In the NIFS systems studied in~\cite{RU} (see Definition and Lemma 2.4 of~\cite{RU}, which makes key use of the open set condition - something we do not impose here), it must be the case that $\pi_\Phi(I^\infty) = J$.  We do not necessarily have this in all cases (see Example~\ref{ExProjMapNotOnto}), but we do note that the additional assumption of the Strong Separation Condition would  allow the proof in~\cite{RU} to apply.  In all cases, however, we do have the following result.

\begin{lemma}\label{LemProjImageClosure}
Let $J'(\Phi)=\{z:\phi_\omega(z)=z \textrm{ for some } \omega \textrm{ in some } I^k \}$ where $\Phi$ is a NIFS on $(U, X)$.  Then $J(\Phi) \subseteq \overline{J'(\Phi)}$, and hence $\overline{J(\Phi)} \subseteq \overline{J'(\Phi)}$.  Also,
\[\overline{J(\Phi)}=\overline{\pi_\Phi(I^\infty)},\]
and so, if $\pi_\Phi(I^\infty)$ is compact, then ${J(\Phi)}={\pi_\Phi(I^\infty)}.$
\end{lemma}

We note that in the non-autonomous case, unlike in the autonomous case (see Claim~\ref{ClaimA'}), $J'$ does not necessarily have to be a subset of $J$, or even of $\overline{J}$.  See Example~\ref{ExHNUP}.

\begin{proof}
Let $z \in J$ and $\delta >0$.  Choose $k$ such that $s^k \textrm{diam}_d(X)< \delta$.  Since $ z \in J \subseteq X_k$, there exists $\omega \in I^k$ such that $z \in \varphi_\omega (X)$. Extend $\omega$ to any $\omega^* \in I^\infty$ and note that, as stated in Remark~\ref{RemPiecesContainLimitPts}, $\varphi_\omega(X)$ contains both the fixed point of the contraction $\varphi_\omega$ and the point $\pi_\Phi(\omega^*) \in J$.  Since $\varphi_\omega (X) \subseteq \Delta_d(z, s^k \textrm{diam}_d(X)) \subseteq \Delta_d(z, \delta)$, we conclude $J \subseteq \overline{J'(\Phi)} \cap \overline{\pi_\Phi(I^\infty)}$.  This and the definition of $\pi_\Phi$ yield that $\overline{J} \subseteq \overline{\pi_\Phi(I^\infty)}\subseteq \overline{J}$.

The final statement follows since if $\pi_\Phi(I^\infty)$ is compact, we have $J(\Phi) \subseteq \overline{J(\Phi)} =\overline{\pi_\Phi(I^\infty)}= \pi_\Phi(I^\infty)\subseteq J(\Phi)$.
\end{proof}

In certain examples, it is convenient to change the set $X$ to a more convenient forward invariant compact set.  The following result shows that such a change to $X$, though it may affect $J$ (see Example~\ref{ExJChanges}), will not affect $\overline{J}$, the central object of study for this paper.

\begin{lemma}\label{LemReplaceX2}
Let $\widetilde{X}\neq \emptyset$ be a compact subset of $X$ that is forward invariant under NIFS $\Phi$ on $(U, X)$.  Then, calling $\widetilde{X}_k:= \bigcup_{\omega \in I^k} \varphi_\omega(\widetilde{X})$, we have \[\overline{J(\Phi)}=\overline{\bigcap_{k=1}^\infty X_k}=\overline{\bigcap_{k=1}^\infty \widetilde{X}_k}.\]

Hence, if each $\widetilde{X}_k$ is compact, then $J(\Phi)=\bigcap_{k=1}^\infty X_k=\bigcap_{k=1}^\infty \widetilde{X}_k.$
\end{lemma}

\begin{proof}
Since, as was noted in Remark~\ref{RemProjMap}, the projection map $\pi_\Phi$ is independent of the choice of non-empty compact forward invariant set $X$, the first result follows immediately from Lemma~\ref{LemProjImageClosure}.

When each $\widetilde{X}_k$ is compact, the second result follows since $J(\Phi) \subseteq \overline{J(\Phi)}=\overline{\bigcap_{k=1}^\infty X_k}=\overline{\bigcap_{k=1}^\infty \widetilde{X}_k}=\bigcap_{k=1}^\infty \widetilde{X}_k \subseteq \bigcap_{k=1}^\infty {X}_k = J(\Phi)$.
\end{proof}

\begin{example}[Projection map $\pi_\Phi:I^\infty \to J$ not onto]\label{ExProjMapNotOnto}
Let $X=[0, 1]$ be the unit interval.  Let $\Phi^{(1)}=\{f_1, f_2, f_3, \dots\}$ where $f_n(z)=\frac{z}3+e_n$ with   $e_n =\frac13-\frac1{3^n}$.  Note that $e_1=0$ and $0<e_n< \frac13$ for all $n \geq 2$.  Let $\Phi^{(k)}=\{f_1\}$ for all $k \geq 2$.

Technically speaking, one should first establish an open set $U \subseteq \C$ (e.g., $\Delta(0,10)$) and corresponding compact subset $X$ (e.g., $\overline{\Delta(0,9)}$) to satisfy the NIFS condition that each function map $U$ into $X$.  And then afterwards use Lemma~\ref{LemReplaceX2} to replace $X$ by the forward invariant interval $[0, 1]$ without altering the limit set $J$.  However, in later examples we forgo such details leaving it for the reader to quickly check that such a procedure can be validly executed.

We now show $\frac13 \in J \setminus \pi_\Phi(I^\infty)$.
Since, for each $n \in \N$, we have $\frac13 \in[e_n, \frac1{3}]=[e_n, \frac1{3^n}+e_n]=f_n\circ f_1^{n-1}(X) \subseteq X_n$, we see $\frac13 \in J$.
However, for each $\omega \in I^\infty$ there must be some $f_n \in \Phi^{(1)}$ such that $\{\pi_\Phi(\omega)\}=\cap_{k=1}^\infty f_n\circ f_1^{k-1}(X)=\cap_{k=1}^\infty [e_n, \frac1{3^k}+e_n]=\{e_n\} \neq \{\frac13 \}$.  Hence $\pi_\Phi(I^\infty)=\{e_n:n \in \N\}$, and so $\pi_\Phi(I^\infty) \neq J =\{e_n:n \in \N\} \cup \{1/3\}$, where the equality follows from Lemma~\ref{LemProjImageClosure}.
\end{example}

\begin{example}[$J$ depends on $X$] \label{ExJChanges}
Let $X=[-1,1]$ and $\widetilde{X}=[0,1]$.  For each $n \in \N$, set $z_n=\frac1{2^n-1}>0$ and $f_n(z)=\frac12 (z-z_n)+z_n$.  Clearly, each of $X$ and $\widetilde{X}$ is forward invariant under each contraction $f_n$.  We consider the (autonomous) system generated where each $\Phi^{(k)}=\{f_n:n \in \N\}$.  Considering $\widetilde{X}_k$ given as in Lemma~\ref{LemReplaceX2}, it is clear that $0 \not \in \widetilde{X}_1$ since, for all $n$, we see $0 \notin [\frac{z_n}{2}, \frac{1+z_n}2]=f_n(\widetilde{X})$.  However, for all $n \in \N$, since the $n$-th iterate $f_n^n(z)=\frac1{2^n} (z-z_n)+z_n$, we see $0 \in [0,f_n^n(1)] = [f_n^n(-1),f_n^n(1)] = f_n^n(X)\subseteq X_n$.
Hence $0 \in \cap_{n =1}^\infty X_n \setminus \cap_{n =1}^\infty \widetilde{X}_n$, showing that $J$ does depend on the choice of forward invariant non-empty compact set $X$ (something which cannot happen in the NIFS systems studied in~\cite{RU} where, as noted,  $J = \pi_\Phi(I^\infty)$ must hold).
\end{example}


\begin{remark}[Invariance Condition]\label{RmkInvariance}
Note that for any $j \geq 1$ and $k \geq 0$, we unpack the relevant definitions (defining each $X_{0}^{(j)}=X$) to see the following invariance condition
\begin{equation}\label{EqInvarCond}
\bigcup_{i \in I^{(j)}} \varphi_i^{(j)}(X_k^{(j+1)})=X_{k+1}^{(j)},
\end{equation}
which is illustrated in Figure~\ref{TablePic} as a way of relating the diagonally adjacent sets $X_{k+1}^{(j)}$ and $X_k^{(j+1)}$.
\end{remark}

\begin{remark}\label{RemInvarCondLimitSets}
Letting $k \to \infty$ in the invariance condition~(\ref{EqInvarCond}) leads one to wonder if we must always have
$\bigcup_{i \in I^{(j)}} \varphi_i^{(j)}(J^{(j+1)})=J^{(j)}$.  While this is not true in general, we do always get the inclusion
\begin{eqnarray*}
\bigcup_{i \in I^{(j)}} \varphi_i^{(j)}(J^{(j+1)})
&=&
\bigcup_{i \in I^{(j)}} \varphi_i^{(j)}(\bigcap_{k=1}^\infty X_k^{(j+1)})
\subseteq \bigcup_{i \in I^{(j)}} \bigcap_{k=1}^\infty \varphi_i^{(j)}( X_k^{(j+1)}) \\
& \subseteq & \bigcap_{k=1}^\infty \bigcup_{i \in I^{(j)}} \varphi_i^{(j)}( X_k^{(j+1)})
=\bigcap_{k=1}^\infty X_{k+1}^{(j)}
=J^{(j)}.
\end{eqnarray*}
Now consider Example~\ref{ExProjMapNotOnto} to see that equality above does not follow.  Since $J^{(2)}=\{0\}$, $\bigcup_{i \in I^{(1)}} \varphi_i^{(1)}(J^{(2)})= \{e_n:n \in \N\} \neq \{e_n:n \in \N\} \cup \{\frac13\}=J^{(1)}$.  Additionally, $\bigcup_{i \in I^{(1)}} \varphi_i^{(1)}\left(\overline{J^{(2)}}\right) \neq \overline{J^{(1)}}$.
\end{remark}

Additional hypotheses, however, lead to the following result.

\begin{lemma}\label{LemInvClosedLimitSets}
Let $\Phi$ be a NIFS on $(U,X)$ and let $j \in \N$.
When $\Phi^{(j)}$ is finite, we have \[\bigcup_{i \in I^{(j)}} \varphi_i^{(j)}\left(\overline{J^{(j+1)}}\right) = \overline{J^{(j)}}.\]  Hence, when $\Phi^{(j)}$ is finite and $J^{(j+1)}$ is compact (e.g., when all $\Phi^{(k)}$, for $k \geq j$, are finite), we see that $\bigcup_{i \in I^{(j)}} \varphi_i^{(j)}\left({J^{(j+1)}}\right) = {J^{(j)}}$.
\end{lemma}

\begin{proof}
To prove the first statement it suffices to consider $j=1$.  Letting $I_1^\infty = \prod_{k=1}^\infty I^{(k)}$ and $I_2^\infty = \prod_{k=2}^\infty I^{(k)}$, we define the respective projection maps $\pi_1:I_1^\infty \to J^{(1)}$ and $\pi_2:I_2^\infty \to J^{(2)}$. We first note that
\[\bigcup_{i \in I^{(1)}} \varphi_i^{(1)}(\pi_2(I_2^\infty))=\pi_1(I_1^\infty),\]
since
\begin{eqnarray*}
  \bigcup_{i \in I^{(1)}} \varphi_i^{(1)}(\pi_2(I_2^\infty)) &=& \bigcup_{i \in I^{(1)}} \varphi_i^{(1)}(\bigcup_{\omega \in I_2^\infty} \{\pi_2(\omega)\})
  =\bigcup_{i \in I^{(1)}} \varphi_i^{(1)}(\bigcup_{\omega \in I_2^\infty} \bigcap_{n=2}^\infty \varphi_{\omega_2 \cdots \omega_n}(X)) \\
  &=& \bigcup_{i \in I^{(1)}}\bigcup_{\omega \in I_2^\infty} \varphi_i^{(1)}(\bigcap_{n=2}^\infty \varphi_{\omega_2 \cdots \omega_n}(X))\\
  &=&\bigcup_{i \in I^{(1)}}\bigcup_{\omega \in I_2^\infty}\bigcap_{n=2}^\infty \varphi_i^{(1)}(\varphi_{\omega_2 \cdots \omega_n}(X)) \\
  &=& \bigcup_{i \in I^{(1)}}\bigcup_{\omega \in I_2^\infty} \bigcap_{n=1}^\infty \varphi_{i \cdot \omega_2 \cdots \omega_n}(X)
=\bigcup_{\omega^* \in I_1^\infty} \{\pi_1(\omega^*)\}
=\pi_1(I_1^\infty),
\end{eqnarray*}
where Lemma~\ref{LemImageDecrCompact} was used with regard to $\varphi_i^{(1)}$ and the decreasing compact sets $\varphi_{\omega_2 \cdots \omega_n}(X)$.

Then, using Lemma~\ref{LemProjImageClosure}, we see
\begin{eqnarray*}
\overline{J^{(1)}}&=&\overline{ \pi_1(I_1^\infty)  }=\overline{\bigcup_{i \in I^{(1)}} \varphi_i^{(1)}(\pi_2(I_2^\infty))}=\bigcup_{i \in I^{(1)}}  \overline{\varphi_i^{(1)}(\pi_2(I_2^\infty))}\\
&=&
\bigcup_{i \in I^{(1)}} \varphi_i^{(1)}\left(\overline{\pi_2(I_2^\infty)}\right)
=\bigcup_{i \in I^{(1)}} \varphi_i^{(1)}\left(\overline{J^{(2)}}\right),
\end{eqnarray*}
where we used the facts that the union is finite, each
$\varphi_i^{(1)}$ is continuous, and the set $\overline{\pi_2(I_2^\infty)}$ is compact.

The final statement of the lemma follows since, if $\Phi^{(j)}$ is finite and $J^{(j+1)}$ is compact, then
${J^{(j)}} \subseteq \overline{J^{(j)}}=\bigcup_{i \in I^{(j)}} \varphi_i^{(j)}\left(\overline{J^{(j+1)}}\right)
=\bigcup_{i \in I^{(j)}} \varphi_i^{(j)}\left({J^{(j+1)}}\right) \subseteq
{J^{(j)}}$, where the last inclusion is justified by Remark~\ref{RemInvarCondLimitSets}.
\end{proof}

\begin{example}\label{ExCantor}
Let $X=[0, 1]$ denote the closed unit interval.  Consider a sequence $(a_j)$ such that each $0< a_j \leq 1/3$, and define maps $\varphi_1^{(j)}(z)=a_j z$ and $\varphi_2^{(j)}(z)=a_j (z-1)+1$.  Then the families of maps $\Phi^{(j)}=\{\varphi_1^{(j)},\varphi_2^{(j)}\}$ define an NIFS.  See Figure~\ref{TablePic}.

\begin{figure}
\centering
  \includegraphics[width=5.4in]{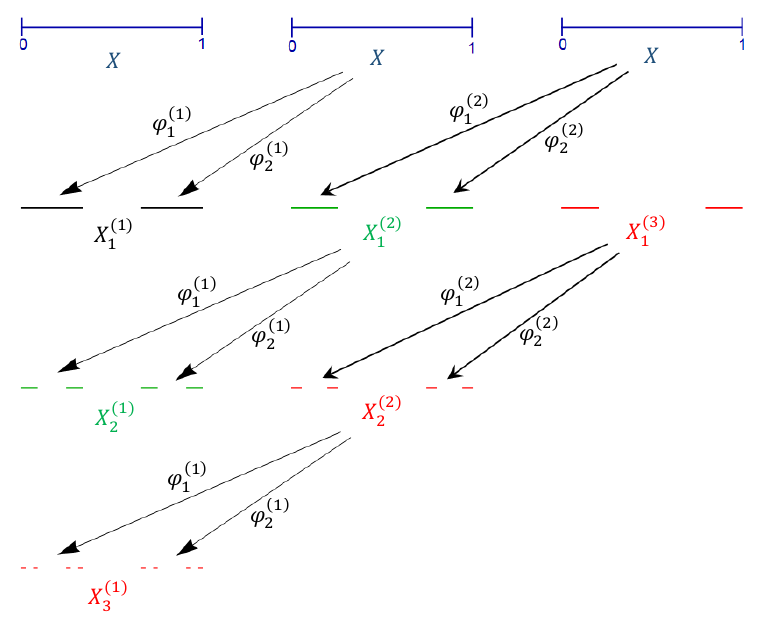}
  \caption{Table illustrating Example~\ref{ExCantor} with $a_1=\frac13, a_2=\frac14$, and $a_3=\frac15$.  Note that sets in each column decrease down to the corresponding limit set, i.e., for each $j \in \N$ we have $\cap_{k=1}^\infty X_k^{(j)}=J^{(j)}$.  Also, note that diagonally adjacent sets $X_{k+1}^{(j)}$ and $X_k^{(j+1)}$ are related by the invariance condition~(\ref{EqInvarCond}) in Remark~\ref{RmkInvariance}. }\label{TablePic}
\end{figure}

\end{example}

\begin{remark}[Combining Stages]\label{RemCombStages}
It will be useful later to analyze a limit set of some NIFS $\Phi$ by first combining stages.  Here we present what this means, in particular, showing that this does not alter the limit set.  First, for families of maps $\Gamma_1, \Gamma_2, \dots, \Gamma_n$, we define $\Gamma_1 \circ \Gamma_2 \circ \cdots \circ \Gamma_n$ to be $\{f_1 \circ f_2 \circ \cdots \circ f_n:f_i \in \Gamma_i\}$.

Given an NIFS $\Phi^{(1)},\Phi^{(2)},\Phi^{(3)}, \dots$ on some $(U,X)$, we can create a new NIFS by combining finite strings of stages as follows.  Consider any strictly increasing sequence $(k_n)_{n=1}^\infty$ of positive integers and define a new NIFS $\widetilde{\Phi}$ by $\widetilde{\Phi}^{(1)}=\Phi^{(1)} \circ \cdots \circ \Phi^{(k_1)}$, $\widetilde{\Phi}^{(2)}=\Phi^{(k_1+1)} \circ \cdots \circ \Phi^{(k_2)}$, and, in general for $n>1$, $\widetilde{\Phi}^{(n)}=\Phi^{(k_{n-1}+1)} \circ \cdots \circ \Phi^{(k_n)}$.

Notice that $\widetilde{\Phi}$ inherits all the defining properties of an NIFS from $\Phi$.  Furthermore, $J(\widetilde{\Phi})=\bigcap_{n=1}^\infty X_{k_n}= \bigcap_{k=1}^\infty X_{k}= J(\Phi)$, since the sets $X_k$ are decreasing.
\end{remark}

\section{Review of Autonomous Attractors}\label{SecAutAttractMainThm}

In this section we review known results for autonomous attractors and relate them to the main results for non-autonomous attractors stated in Theorems~\ref{Thm2Main} and~\ref{ThmAnalNIFS}.

The system $\Phi$ in Definition~\ref{defNIFS} is called \emph{autonomous} (and thus just called an IFS) if $I^{(j)}$ and $\Phi^{(j)}$ are independent of $j$, i.e., each $\Phi^{(j)}=\{g_i:i \in I\}$ for some index set $I$.  In such an instance we use the notation $A$ for the attractor instead of $J$ in order to give a notational reminder that we are in a very special (and previously well-studied) case.  For such an autonomous system, we let $G=\langle g_i:i \in I
\rangle$ denote the set of all finite compositions of
generating maps $\{g_i:i \in I\}$, and, following~\cite{RS4}, simply say $G=\langle g_i:i \in I \rangle$ is an IFS on $(U,X)$.

\begin{claim}\label{ClaimA'}
When $\Phi$ is autonomous,
the attractor set $A=J$ given in Definition~\ref{LimitSetdef} satisfies $A \supseteq A'$ and $\overline{A}=\overline{A'}$, the closure of $A'$ in the Euclidean topology (equivalently given by the metric $d$), where
$A'=A'(G):=\{z:\textrm{there exists } g
\in G \textrm{ such that } g(z)=z
\}$ is the set of (attracting) fixed points of $G$.
\end{claim}

Note that in~\cite{RS4} the attractor set was \textit{defined} to be $\overline{A'}$ and not defined in terms of $X_k$ as in Definition~\ref{LimitSetdef}.  This claim, however, shows that the closures of the sets given by the two definitions yield the same set.

\begin{proof}
Let $z \in A'$.  Since the system is autonomous,  there exist some $k \in \N$ and $\omega \in I^k$ such that $\phi_\omega(z)=z$.  Clearly then for each $n$ we see that $z \in \phi_\omega^n(X) \in X_{kn}$, where $\phi_\omega^n$ denotes the $n$th iterate of $\phi_\omega$ (note that the autonomous condition is used here).  Hence $z \in \cap_{n=1}^\infty X_{kn} =\cap_{k=1}^\infty X_{k}=J=A$.  Thus $A' \subseteq A$, and so $\overline{A'} \subseteq \overline{A}$.

The reverse inclusion follows from Lemma~\ref{LemProjImageClosure}.
\end{proof}

If each $\Phi^{(j)}=\{g_1, \dots, g_N\}$, a situation we call the \emph{finite autonomous} case, then the attractor $A$ is the unique non-empty compact subset of $U$ that has the \textit{self-similarity} property given by
\begin{equation}\label{invariance}
A = \bigcup_{i=1}^N g_i(A)
\end{equation}
(see~\cite{Hu}, p.~724).  We note that in this finite autonomous case, the sets $X_k^{(j)}$, and $J^{(j)}$ are all independent of $j$ (in Example~\ref{ExCantor} illustrated in Figure~\ref{TablePic} this would amount to sets across rows being identical because $a_1=a_2=a_3=\dots$).  Furthermore, the  invariance shown in Remark~\ref{RmkInvariance} then becomes $\cup_{i=1}^N g_i(X_k)=X_{k+1}$, which by taking the limit as $k \to \infty$ in a suitable space produces~(\ref{invariance}) (see~\cite{Hu} or apply Lemma~\ref{LemInvClosedLimitSets}).

\begin{remark}
We also point out that in~\cite{MU1,MU2} the
limit set $J$ of a \textit{conformal} IFS is defined a bit differently, but with a clear connection to our definition. See~\cite{MU1,MU2} for a
discussion on the Hausdorff dimension, packing dimension, and other
properties of limit sets of their conformal IFS's.
\end{remark}


In~\cite{RS3} certain \emph{autonomous conformal} attractor sets are shown to be
uniformly perfect, when the generating maps are M\"obius.
Then in~\cite{RS4} a collection of results regarding uniform perfectness are given for \emph{autonomous analytic} attractor sets.  The motivation for the current paper is to explore to what degree, if any, these results generalize to the non-autonomous case.  Hence we first state the major results from~\cite{RS4}.

\begin{othertheorem}[Corollary 1.1 in~\cite{RS4}]\label{cor3}
Let $G=\langle g_i:i \in I \rangle$ be an analytic IFS on $(U,X)$ such
that there exists $\eta>0$ where $|g_i'| \geq \eta$ on $A$ for all $i \in I$.  If $A$  has infinitely many points,
then $A$ is uniformly perfect.
\end{othertheorem}

\begin{othertheorem}[Corollary 1.2 in~\cite{RS4}]\label{main}
Let $G=\langle g_i:i \in I \rangle$ be a conformal IFS on $(U,X)$ such
that there exist $\eta>0$ where $|g_i'| \geq \eta$ on $A$ for all $i \in I$.
If $A$ contains more than one point, then $A$ is
uniformly perfect.
\end{othertheorem}

\begin{othertheorem}[Corollary 1.3 in~\cite{RS4}]\label{cor}
Let $G=\langle g_1, \dots, g_N \rangle$ be a conformal IFS on $(U,X)$.
If $A$ contains more than one point, then $A$ is uniformly perfect.
\end{othertheorem}

The proofs of Theorems~\ref{cor3}-\ref{cor} in~\cite{RS4}, which consider only autonomous systems, heavily rely on the facts (i) $A' \subseteq A$, and (ii) $A$ is forward invariant under $G$, i.e., for every $a \in A$ and $g \in G$ we have $g(a) \in A$ (Lemma 2.2 in~\cite{RS4}).  The main complicating features of the non-autonomous systems we wish to consider in this paper are that these properties do not hold or generalize in a way that allows for the techniques in~\cite{RS4} to be easily adapted to such more general systems (see Example~\ref{ExHNUP} and Remark~\ref{RemA'NotInA}).  Here, however, we do prove Theorem~\ref{Thm2Main} regarding conformal NIFS's and Theorem~\ref{ThmAnalNIFS} regarding analytic NIFS's.

Section~\ref{SecExamples} presents some examples to demonstrate why the possible generalizations of Theorems~\ref{cor3}-\ref{cor} do not hold for general NIFS's, in particular, showing that both (i) and (ii) can fail.

\section{Definitions and basic facts}\label{SectDefsAndFacts}

The main object of interest to this paper is the \textit{analytic} NIFS.  This allows us, via the next result used similarly in~\cite{RS4}, to employ the hyperbolic metric in the definition of NIFS.  In particular, any sequence $\Phi^{(1)},\Phi^{(2)},\Phi^{(3)}, \dots$, such that each $\Phi^{(j)}$ is a collection of non-constant \textit{complex analytic} functions $(\varphi_i^{(j)}:U \to X)_{i \in I^{(j)}}$, where each
function maps the non-empty open connected set $U
\subset \C$ into a compact set $X \subset U$,  will automatically be uniformly contracting with respect to the hyperbolic metric on $U$.
Note that $U \subseteq \C$ must support a hyperbolic metric since $U$ cannot be the plane or punctured plane else the image of $U$ under a non-constant analytic map would have to be dense in $\C$.

\begin{lemma}[Lemma 2.1 of~\cite{RS4}]\label{hypmetric}
If the analytic function $\varphi$ maps an open connected set $U \subset
\C$ into a compact set $X \subset U$, then there exists $0<s<1$,
which depends on $U$ and $X$ only, such that $d(\varphi(z),\varphi(w))\leq s
d(z,w)$ for all $z, w \in X$ where $d$ is the
hyperbolic metric defined on $U$.
\end{lemma}

\begin{remark}\label{RemConnX}
Let $\Phi$ be an analytic NIFS on $(U,X)$.  Note that, for each $x \in X$, the hyperbolic disk $\Delta_d(x, 2\cdot \textrm{diam}(X)) \subset U$ contains $X$ and is connected (being the continuous image of a connected hyperbolic disk in $\Delta(0,1)$).  Hence, $\widetilde{X}=\overline{\bigcup_{x \in X} \Delta_d(x,2\cdot \textrm{diam}(X))}$ is connected (and compact).  Further, since $X$ is forward invariant under $\Phi$, then so is $\widetilde{X}$ since analytic maps cannot increase hyperbolic distances.  We note then that Lemma~\ref{LemReplaceX2} (with the roles of $X$ and $\widetilde{X}$ reversed) allows us to replace $X$ by the connected $\widetilde{X}$ without altering $\overline{J}$.
\end{remark}

We call a doubly connected domain $A$ in $\C$ that can be conformally mapped onto a true (round) annulus $\mathrm{Ann}(w;r,R)=\{z:r < |z-w| < R\}$, for some $0<r<R$, a \emph{conformal annulus} with the \emph{modulus} of $A$ given by $\mathrm{mod~}A=\log(R/r)$, noting that $R/r$ is uniquely determined by $A$ (see, e.g., the version of the Riemann mapping theorem for multiply connected domains in~\cite{A}).

\begin{definition} \label{sepann}
A conformal annulus $A$ is said to
\emph{separate} a set $F \subset \C$ if $F \cap A = \emptyset$ and $F$ intersects both components of $\C \setminus
A$.
\end{definition}

\begin{definition} \label{updef}
A compact subset $F \subset \C$ with two or more points is \emph{uniformly perfect} if
there exists a uniform upper bound on the modulus of each conformal annulus which separates $F$.
\end{definition}

\begin{remark}\label{RemTrueAnn}
Because of the following well-known lemma (see, e.g., Theorem~2.1 of~\cite{McM}), we can equivalently characterize uniformly perfect sets in terms of only true annuli:  A compact subset $F \subset \C$ with two or more points is \emph{uniformly perfect} if there exists a uniform upper bound on the modulus of each \textit{true} annulus (centered at a point in $F$, if we choose) which separates $F$.
\end{remark}

\begin{lemma}\label{LemRoundAnn}
Any conformal annulus $A \subset \C$ of sufficiently large modulus contains an essential true annulus $B$ (i.e., $B$ separates the boundary of $A$) with $\mathrm{mod~}A=\mathrm{mod~}B + O(1)$.  Since, for any $R>3r$ and any $w' \in \overline{\Delta(w,r)}$, the true annulus $\mathrm{Ann}(w';2r, R-r)$ is an essential annulus of $\mathrm{Ann}(w;r, R)$, we may choose $B$ to be centered at any given point in the bounded component of $\C \setminus A$.
\end{lemma}

\begin{remark}\label{RmkAnnInf}
For the case when the conformal annulus $A \subset \CC$ contains infinity the above lemma can be modified to read as: Any conformal annulus $A \subset \CC$ of sufficiently large modulus contains an essential true annulus $B$ (i.e., $B$ separates the boundary of $A$) with $\mathrm{mod~}A=2\mathrm{mod~}B + O(1)$.  To see this note that $A$ contains two disjoint essential conformal annuli $A'$ and $A''$ each with half the modulus of $A$, at most one of which, say, $A''$ can contain infinity.  This can be observed by considering mapping $A$ conformally onto $\mathrm{Ann}(0;1,R)$, and then taking the preimages of $\mathrm{Ann}(0;1,\sqrt{R})$
and $\mathrm{Ann}(0;\sqrt{R},R)$ inside of $A$.  By applying Lemma~\ref{LemRoundAnn} to $A'$ we can obtain our desired result.
\end{remark}

Recall that a compact set $E \subset \C$ is called \textit{hereditarily non uniformly perfect} (HNUP) if no subset of $E$ is uniformly perfect.  Often a set is shown to be HNUP by showing it satisfies the following stronger property of \emph{pointwise thinness}.  This is done in several examples in~\cite{SSS}, and will be done in Example~\ref{ExHNUP}.  Also, certain non-autonomous Julia sets in~\cite{CSS} and in Theorem~\ref{ThmDichotomy} are shown to be HNUP this way (where it is worth noting that the Julia sets constructed are limit sets of conformal NIFS's).


\begin{definition}\label{DefPointThin}A set $E \subset \C$ is \textit{pointwise thin at} $z \in E$ if there exists a sequence of conformal annuli $A_n$ each of which separates $E$, has $z$ in the bounded component of its complement, and such that $\textrm{mod }A_n \to +\infty$ while the Euclidean diameter of $A_n$ tends to zero.  A set $E \subset \C$ is called \emph{pointwise thin} when it is pointwise thin at each of its points.
\end{definition}

Note that any pointwise thin compact set is HNUP since none of its points can lie in a uniformly perfect subset.
Also note that if $E$ is pointwise thin, then $\overline{E}$ is pointwise thin at each point of $E$ (but not necessarily pointwise thin at each point of $\overline{E}$ as the next example illustrates).

\begin{example}[Closure of pointwise thin is not pointwise thin]\label{ExClosurePTnotPT}
The set $E=\{2^{-n}:n \in \N\}$ is trivially pointwise thin, but its closure $\overline{E}$ is not pointwise thin at 0 since the reader can check that the modulus of any round annulus separating $\overline{E}$ and containing 0 must be bounded by $\log 2$.
\end{example}

\begin{lemma}\label{LemmaAnn2}
Suppose $A=\mathrm{Ann}(z;r, R)$, for some $z\in \C$ and $0<r<R$, is a true annulus separating $J$, where $J =\cap_{k=1}^\infty X_k$ is the attractor of some NIFS $\Phi$.  Fix $0<\delta < \frac{R-r}2$.  Then the annulus $B=\mathrm{Ann}(z;r+\delta, R-\delta) \subset A$ separates some $X_k$.  Hence, given any $0<\epsilon < \mathrm{mod~}A$, we can choose $\delta>0$ such that $\mathrm{mod~} B = \log(\frac{R-\delta}{r+\delta})= \log (\frac{R}{r}) - \epsilon=\mathrm{mod~} A-\epsilon$, where $B$ separates some $X_k$.
\end{lemma}

\begin{proof}
Since $A$ separates $J$ and $B$ is an essential subannulus of $A$, both components of $\C \setminus B$ must meet $J$, and therefore must meet each $X_k \supseteq J$. We complete the proof by showing that $B \cap X_k = \emptyset$ for some $k$.  Suppose not.
Now fix $k$ and choose $z_k \in X_k \cap B$.  Hence there exists $\omega \in I^k$ such that $z_k \in  \varphi_\omega(X)$.  Since $\mathrm{diam}_d(\varphi_\omega(X))\leq s^k\mathrm{diam}_d(X)$ (see Remark~\ref{RemPiecesContainLimitPts}), we have that $\varphi_\omega(X) \subseteq \Delta_d(z_k, s^k\mathrm{diam}_d(X)) \subset A$ for $k$ sufficiently large (since $z_k \in \overline{B} \subset A$ and $d$ generates the Euclidean topology on $X$).  Since $\varphi_\omega(X)$ must contain a point of $J$ (see Remark~\ref{RemPiecesContainLimitPts}), we see that $A \cap J \neq \emptyset$ and thus $A$ does not separate $J$, which is a contradiction.
\end{proof}

\begin{lemma}\label{LemRbound}
Suppose $A= \mathrm{Ann}(z;r,R)$, for some $z\in \C$ and $0<r<R$, separates $E \subseteq X \subset \C$ where $\mathrm{diam}(X)<\infty$ and $R \geq 2 \cdot \mathrm{diam}(X)$.  Then $\frac{R}{r} \leq 2$.
\end{lemma}

\begin{proof} Since $A$ separates $E$, there exist $x_1, x_2 \in E$ with $|x_1-z|\geq R$ and $|x_2-z|\leq r$.  Hence $2 \cdot \mathrm{diam}(X) - r \leq R-r \leq |x_1-x_2| \leq \mathrm{diam}(E) \leq \mathrm{diam}(X)$, which gives that $\mathrm{diam}(X) \leq r$.  Again using that $R-r \leq \mathrm{diam}(X)$, we see that $\frac{R-r}{r} \leq \frac{\mathrm{diam}(X)}{r} \leq 1$, which gives $\frac{R}{r}  \leq 2$ as desired.
\end{proof}

The following is a result that seems to be well understood by many but, since a reference could not be found, we provide a proof here.

\begin{proposition}\label{PropUPImage}
Let $f:U \to \C$ be non-constant and analytic on open connected $U \subset \C$.  Suppose that $E \subset U$ is uniformly perfect.  Then $f(E)$ is uniformly perfect.
\end{proposition}

This result follows from the fact that \textit{locally} non-constant analytic maps are either conformal or behave like $z \mapsto z^k$ for some $k \in \N$, which can distort the modulus of an annulus by at most a factor of $k$.

\begin{proof}
The local behavior of non-constant analytic maps clearly implies that, since $E$ is perfect, so is $f(E)$.
We now suppose towards a contradiction that $f(E)$ is not uniformly perfect.  Hence there exists true annuli $A_n=\mathrm{Ann}(w_n;r_n, R_n)$ which separate $f(E)$ with $R_n/r_n \to \infty$.

By Lemma~\ref{LemRoundAnn}, we may
assume each $w_n \in f(E)$.  Since $f(E)$ is perfect, it follows that $R_n \to 0$ (see, e.g., Lemma 2.7 of~\cite{RS4}).

By compactness of both $f(E)$ and $E$, and passing to a subsequence if necessary, we may assume there exists $w_0 \in f(E)$ such that  $w_n \to w_0$ and $z_0, z_n \in E$ such that $z_n \to z_0$ with each $f(z_n)=w_n$.

Suppose $f'(z_0) \neq 0$.  Thus there exists a local branch $h$ of $f^{-1}$ defined on some neighborhood of $w_0$.  Hence, the conformal annuli $h(A_n)$, for large $n$, must then separate $E$, which is a contradiction since $E$ is uniformly perfect and $\mathrm{mod~}h(A_n)=\mathrm{mod~}A_n \to \infty$.

Now suppose $f'(z_0) = 0$, and choose $k$ such that $f$ maps $z_0$ to $w_0$ with multiplicity $k>1$.  By pre- and post- composing with translations, we may assume $z_0=w_0=0$, and so there exists a conformal map $g$ defined on a neighborhood of $0$ such that $gfg^{-1}(z)=z^k$ (see, e.g., Theorem 6.10.1 of~\cite{Be}).  It suffices to consider two cases:  Case(i) Each $A_n$ surrounds $w_0=0$, and Case (ii) No $A_n$ surrounds $w_0=0$.

\textbf{Case (i):}  From each conformal annulus $g(A_n)$ of large modulus (and so for all large $n$), we apply Lemma~\ref{LemRoundAnn} to extract an essential true annulus $B_n=\mathrm{Ann}(0;s_n,S_n)\subseteq g(A_n)$ of modulus $\mathrm{mod~}B_n=\mathrm{mod~}A_n -K$, for some fixed $K>0$.  Since $A_n'=\mathrm{Ann}(0;s_n^{1/k},S_n^{1/k})$ maps by $z \mapsto z^k$ onto $B_n \subseteq g(A_n)$, we must have that each conformal annulus $g^{-1}(A_n')$ surrounds $z_0=0$ and $\mathrm{mod~}g^{-1}(A_n')=\mathrm{mod~}(A_n') =\frac{1}{k} \mathrm{mod~}B_n \to \infty$, which is a contradiction since each $g^{-1}(A_n')$ separates the uniformly perfect set $E$.

\textbf{Case (ii):}  Again for each conformal annulus $g(A_n)$ of large modulus (and so for all large $n$), we apply Lemma~\ref{LemRoundAnn} to extract an essential true annulus $B_n=\mathrm{Ann}(g(w_n);s_n,S_n)\subseteq g(A_n)$ of modulus $\mathrm{mod~}B_n=\mathrm{mod~}A_n -K$, for some fixed $K>0$.  Note that no $\Delta(g(w_n),S_n)$ contains $0$.
Hence, the map $z \mapsto z^k$ has $k$ well-defined inverse branches on $B_n$, one of which must map $B_n$ to a conformal annulus $B_n'$ surrounding $g(z_n)$.  And so, $g^{-1}(B_n')$ is a conformal annulus surrounding $z_n$ and separating $E$, with modulus  $\mathrm{mod~}g^{-1}(B_n')=\mathrm{mod~}B_n'=\mathrm{mod~}B_n=\mathrm{mod~}A_n -K$.  This is a contradiction since $E$ is uniformly perfect and $\mathrm{mod~}A_n \to \infty$.
\end{proof}

The following result can be proven using the style of argument used to prove Proposition~\ref{PropUPImage}, and so we omit the details.

\begin{proposition}\label{PropPointThinImage}
Let $f:U \to \C$ be non-constant and analytic on open connected $U \subset \C$.  Suppose that compact $E \subset U$ is pointwise thin at $z \in E$ and $z$ is the only point of $E$ which maps to $f(z)$.  Then $f(E)$ is pointwise thin at $f(z) \in f(E)$.
\end{proposition}

The following example shows that in the above hypothesis it is critical that $z$ is the only point of $E$ which maps to $f(z)$.
\begin{example}[Analytic image of pointwise thin is not pointwise thin]
Letting each $x_n=2^{-n^2}$, we set $E_1=\cup_{n \geq 0} [x_{2n+1},x_{2n}] \cup \{0\}$ and $E_2=\cup_{n \geq 0} [x_{2n+2},x_{2n+1}] \cup \{0\}$.  By considering the annuli $\mathrm{Ann}(0;x_{2n+2},x_{2n+1})$, one can show that
 $E_1$ is pointwise thin at 0.  Similarly, $E_2$ can be shown to be pointwise thin at 0. Note also that $E_1 \cup E_2 = [0,1]$.  Hence both $F_1=E_1 \cap (\{2^{-n}:n \in \N\} \cup \{0\})$ and $F_2=E_2 \cap (\{2^{-n}:n \in \N\} \cup \{0\})$ are pointwise thin (at each point).  Using the principal branch to define $g(z)=\sqrt{z+1}$, we see by Proposition~\ref{PropPointThinImage} that $g(F_1)$ is pointwise thin.  Similarly, we consider $-g(z)=-\sqrt{z+1}$ to show that $-g(F_2)$ is pointwise thin.  Note that $g$ and $-g$ are branches of the inverse of $f(z)=z^2-1$.  Letting $E = g(F_1) \cup -g(F_2)$, we then see that $E$ is compact and pointwise thin (and thus pointwise thin at both 1 and $-1$), but $f(E)=F_1 \cup F_2=\{2^{-n}:n \in \N\} \cup \{0\}$ is not pointwise thin at $f(1)=f(-1)=0$ (as noted in Example~\ref{ExClosurePTnotPT}).
\end{example}

The following result can easily be shown.

\begin{lemma}\label{LemImageDecrCompact}
Suppose $f: X \to Y$ is continuous and compact sets $A_n \subseteq X$ form a decreasing sequence.  Then $f(\cap_{n=1}^\infty A_n)=\cap_{n=1}^\infty f(A_n)$.
\end{lemma}

We close this section with a remark that relates to Proposition~\ref{PropPointThinImage}.

\begin{remark}  In~\cite{Shiga}, Shiga discusses the
quasiconformal equivalence of Cantor sets which appear as limits sets of non-autonomous IFSs and Julia sets of rational maps.
Many complex analysts are interested in the complements of various kind of Cantor sets, since the complements of Cantor sets are good examples of Riemann surfaces of infinite type.

We note the following:

(a)  If a compact set $K$ in the plane is uniformly perfect and a compact set $L$ in the plane is not uniformly perfect, then there is no quasiconformal map $g: \Bbb{C}\rightarrow \Bbb{C}$ such that $g(K)=L.$

(b) It is an open problem whether there are multiple quasiconformal equivalence classes of pointwise thin limit sets of non-autonomous IFSs.
\end{remark}

\section{Examples}\label{SecExamples}

In this section we provide examples to show that possible generalizations of Theorems~\ref{cor3}-\ref{cor} of Section~\ref{SecAutAttractMainThm} to the non-autonomous case do not hold. Specifically, we show that none of the following Statements 1-3 hold.  Examples to illustrate Theorem~\ref{Thm2Main} and Corollary~\ref{CorHNUP} are also given, along with an analysis of how these results generalize Theorem 4.1 of~\cite{SSS}.

\textbf{Statement 1:} (Generalization of Theorem~\ref{cor3})
Let $\Phi^{(1)}, \Phi^{(2)}, \dots$ be an analytic NIFS on $(U,X)$ such
that there exists $\eta>0$ with $|\varphi'| \geq \eta$ on $X$ for all $\varphi \in \cup_{j=1}^\infty \Phi^{(j)}$.  If $J$  has infinitely many points,
then $J$ is uniformly perfect.

\vskip.1in

\textbf{Statement 2:} (Generalization of Theorem~\ref{main})
Let $\Phi^{(1)}, \Phi^{(2)}, \dots$ be a conformal NIFS on $(U,X)$ such
that there exists $\eta>0$ with $|\varphi'| \geq \eta$ on $X$ for all $\varphi \in \cup_{j=1}^\infty \Phi^{(j)}$.
If $J$ contains more than one point, then $J$ is
uniformly perfect.

\vskip.1in

\textbf{Statement 3:} (Generalization of Theorem~\ref{cor})
Let $\Phi^{(1)}, \Phi^{(2)}, \dots$ be a conformal NIFS on $(U,X)$ such that there is a uniform bound on the cardinality of $\Phi^{(j)}$.
If $J$ contains more than one point, then $J$ is uniformly perfect.

\vskip.1in

\begin{example}\label{ExHNUPcantor}
Each set $I_{\bar{a}}$ in Theorem 4.1 of~\cite{SSS} is a limit set of a NIFS suitably chosen as follows.  Set $X=[0,1]$, fix $m \in \{2, 3, \dots\}$, and choose $0< a \leq \frac{1}{m+1}$.  Fix a sequence $\bar{a}=(a_1, a_2, \dots)$ such that $0<a_k \leq a$ for $k=1, 2, \dots$.  For each $k \in \N$, set $\Phi^{(k)}$ to be the collection $\{\varphi_1^{(k)}, \dots, \varphi_m^{(k)}\}$ of linear maps, each with derivative $a_k$, such that the images $\varphi_1^{(k)}(X), \dots, \varphi_m^{(k)}(X)$ are $m$ equally spaced subintervals of $X$ with $\varphi_1^{(k)}(X)=[0, a_k]$ and $\varphi_m^{(k)}(X)=[1-a_k, 1]$.  Example~\ref{ExCantor}, illustrated in Figure~\ref{TablePic}, is such an NIFS (with $m=2$).  Each set $X_k$ then coincides with what~\cite{SSS} calls $I_k$, and consists of $m^k$ basic intervals.  And the limit set $J$ then coincides with what~\cite{SSS} calls $I_{\bar{a}}$.

Theorem 4.1(1) of~\cite{SSS} shows that $J$ is perfect, but pointwise thin (and thus HNUP) when $\liminf a_k =0.$  We now show that this also follows from Corollary~\ref{CorHNUP}.  In order to use this corollary we set $U=\Delta(\frac12, 0.7)$ and $X=\overline{\Delta}(\frac12,0.6)$, recalling that Lemma~\ref{LemReplaceX2} shows that $J$ is unchanged by this change of $X$ from $[0, 1]$.  Selecting a subsequence $a_{k_n} \to 0$, the reader can quickly check that $\inf_n\{b_{k_n}\}>0, \inf_n\{\delta_{k_n}\}>0$, and $\eta_{k_n} =a_{k_n}\cdot \mathrm{diam}(X) \to 0$, and thus Corollary~\ref{CorHNUP} applies (since $\Phi$ clearly satisfies the Strong Separation Condition).
We also note that when $\liminf a_k =0$, Corollary~\ref{CorHNUP} shows $J$ is pointwise thin even when the strict setup above is considerably relaxed (e.g., the sets $\varphi_1^{(k)}([0, 1]), \dots, \varphi_m^{(k)}([0, 1])$ do not need to be \textit{equally spaced} subintervals of $[0, 1]$).

Theorem 4.1(2) of~\cite{SSS} shows that $J$ is uniformly perfect when $\liminf a_k >0$.  This also follows from Theorem~\ref{Thm2Main}, noting that we may choose $\eta = \inf a_k >0$ to satisfy the Derivative Condition and choose $\delta = 1-2a$ to satisfy the Two Point Separation Condition (even when $\liminf a_k =0$) since the images $\varphi_1^{(k)}(X)$ and $\varphi_m^{(k)}(X)$ are always a distance $1-2a_k$ apart.
We also note that when $\liminf a_k >0$, Theorem~\ref{Thm2Main} shows $J$ is uniformly perfect even when the strict setup above is considerably relaxed.  For example, the sets $\varphi_1^{(k)}(X), \dots, \varphi_m^{(k)}(X)$ do not need to be \textit{equally spaced} subintervals of $X$.  In fact, these sets could even overlap, as long as the Two Point Separation Condition is met (and $\liminf a_k >0$), and $J$ would still be uniformly perfect.
\end{example}

\begin{remark}\label{RemDerivCritical}
Note that Example~\ref{ExCantor}, with each $a_j=\frac{1}{j+2}$, shows that Statement 3 does not hold since $J$ would then be perfect but also be HNUP. It also illustrates that the Derivative Condition in Theorem~\ref{Thm2Main} is critical, even when all the other conditions are met.
\end{remark}

\begin{example}\label{ExHNUP}
Again, let $X=[0, 1]$.  Set $f_1(z)=\frac{z}3, f_2(z)=\frac{z+2}3$ and $f_3(z)=\frac13 (z-\frac12)+\frac12$.  We fix a sequence of postive integers $(l_j)$, and then create $\Phi$ by choosing $\Phi^{(1)}=\{f_1,f_2\}, \Phi^{(2)}=\Phi^{(3)}=\dots=\Phi^{(1+l_1)}=\{f_3\},\Phi^{(1+l_1+1)}=\{f_1,f_2\},\Phi^{(1+l_1+2)}=\Phi^{(1+l_1+3)}=\dots=
\Phi^{(1+l_1+1+l_2)}=\{f_3\}$, etc.  Hence, defining $L_0=0$ and $L_n=\sum_{j=1}^n(1+l_j)$, we have, for each $n=0, 1, 2, \dots$, $\Phi^{(L_n+1)}=\{f_1,f_2\}$ and $\Phi^{(L_n+i)}=\{f_3\}$ for $2\leq i \leq 1+l_{n+1}$.

We prove the following dichotomy.
\vskip.05in

\textbf{Claim:} We have that $\sup l_j =+\infty$ implies $J$ is perfect but pointwise thin (and thus HNUP), whereas $\sup l_j <+\infty$ implies $J$ is uniformly perfect.

We now consider a related NIFS $\widetilde{\Phi}$ such that $J(\widetilde{\Phi})=J(\Phi)$ by combining stages of consecutive $\Phi^{(j)}$ which equal $\{f_3\}$ (see Remark~\ref{RemCombStages}).  Specifically, we have $\widetilde{\Phi}^{(1)}=\Phi^{(1)}=\{f_1,f_2\}, \widetilde{\Phi}^{(2)}=\Phi^{(2)} \circ \Phi^{(3)} \circ \dots \circ \Phi^{(l_1 + 1)}=\{f_3^{l_1}\}, \widetilde{\Phi}^{(3)}=\Phi^{(1+l_1+1)}=\{f_1,f_2\}, \widetilde{\Phi}^{(4)}=\{f_3^{l_2}\}, \dots$, noting each iterate $f_3^{l_n}(z)=\frac1{3^{l_n}} (z-\frac12)+\frac12$.  More succinctly we have for each $n \in \N$, $\widetilde{\Phi}^{(2n-1)}=\{f_1,f_2\}$ and $\widetilde{\Phi}^{(2n)}=\{f_3^{l_n}\}$.
We now replace $\Phi$ by  $\widetilde{\Phi}$, hence the $X_n^{(j)}$ and $I^j$ below formally are constructed in reference to $\widetilde{\Phi}$ (see Figure~\ref{TablePicHNUP}).

\begin{figure}
  \centering
  \hskip-.2in
  \includegraphics[width=5.2in]{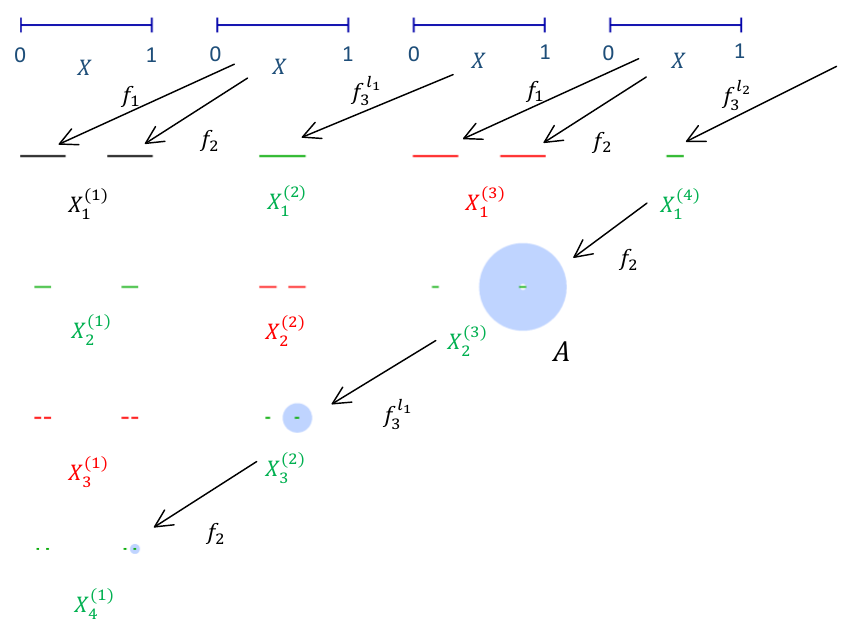}
  \caption{Table illustrating $\widetilde{\Phi}$ in Example~\ref{ExHNUP}, where $l_1=1$ and $l_2=2$.}\label{TablePicHNUP}
\end{figure}

%
%
%

We now suppose $\sup l_j <+\infty$ and prove $J$ is uniformly perfect.  Again we combine stages, this time doing so in order to utilize Theorem~\ref{Thm2Main}.  Create NIFS $\Psi$ with $J(\Psi)=J(\widetilde{\Phi})=J(\Phi)$ by stipulating that, for each $k \in \N$,  $\Psi^{(k)}=\widetilde{\Phi}^{(2k-1)}\circ \widetilde{\Phi}^{(2k)}=\{f_1 \circ f_3^{l_k},f_2 \circ f_3^{l_k}\}$.   Since the images $f_1 \circ f_3^{l_k}(X) \subseteq f_1(X)=[0,1/3]$ and $f_2 \circ f_3^{l_k}(X) \subseteq f_2(X)=[2/3, 1]$ are always separated by $\delta=1/3$, we see that the Two Point Separation Condition (with respect to $\Psi$)  is met.  Further the Derivative Condition (with respect to $\Psi$) is also met (when $\sup l_j <+\infty$, but not when $\sup l_j =+\infty$) since each map in $\Psi^{(k)}$ is linear with derivative $\frac{1}{3^{l_k+1}}$.  From Theorem~\ref{Thm2Main} it then follows that $J(\Psi)$ is uniformly perfect.

We now suppose that $\sup l_j =+\infty$ in order to show $J(\Phi)=J(\Psi)$ is perfect but pointwise thin. Perfectness follows from the fact that the diameter of each component of $X^{(1)}_{2n}$ shrinks to zero as $n \to \infty$ and each component of $X^{(1)}_{2n}$ contains two components of $X^{(1)}_{2n+2}$.  Now note that we may take $\Psi$ to be an NIFS on $(U,\widetilde{X})$ with $U=\Delta(\frac12, 0.7)$ and $\widetilde{X}=\overline{\Delta}(\frac12,0.6)$  Select a subsequence $l_{k_n} \to \infty$.  Since the images $f_1 \circ f_3^{l_k}(\widetilde{X}) \subseteq f_1(\widetilde{X})\subset \mathrm{Int}(\widetilde{X})$ and $f_2 \circ f_3^{l_k}(\widetilde{X}) \subseteq f_2(\widetilde{X}) \subset \mathrm{Int}(\widetilde{X})$, the reader can quickly check that $\Psi$ clearly satisfies the Strong Separation Condition and $\inf_n\{b_{k_n}\}>0, \inf_n\{\delta_{k_n}\}>0$, and $\eta_{k_n} =\frac{\mathrm{diam}(\widetilde{X})}{3^{l_{k_n}+1}}\to 0$ (since each map in $\Psi^{(k_n)}$ is linear with derivative $\frac{1}{3^{l_{k_n}+1}}$).  Hence, Corollary~\ref{CorHNUP} applies to show $J(\Phi)=J(\Psi)$ is pointwise thin.

\end{example}


\begin{remark}\label{RemA'NotInA}
Example~\ref{ExHNUP} shows that (when $\sup l_j =+\infty$) $J(\Phi)$ can be perfect yet fail to be uniformly perfect even when $\Phi$ (but not the modified NIFS $\widetilde{\Phi}$) satisfies both the Derivative Condition and M\"obius Condition of Theorem~\ref{Thm2Main}.  This example shows that the Two Point Separation Condition in Theorem~\ref{Thm2Main} is critical, and also shows that none of the above Statements 1-3 hold.  We also note that $J'=\{z:\phi_\omega(z)=z \textrm{ for some } \omega \textrm{ in some } I^k \}$ is not a subset of $J$ (e.g., 0 is a fixed point of $f_1$ but is not in $J$).  Hence, also $J$ is not forward invariant under the maps $\phi_\omega$  for $\omega \in I^k$.  Compared with statements (i) and (ii) as given for autonomous IFSs near the end of Section~\ref{SecAutAttractMainThm}, we note that the non-autonomous situation is far more delicate.
\end{remark}

\section{Applications to Non-Autonomous Julia Sets}\label{SectApplicationsJulia}
Given a sequence of complex polynomials $(P_j)$, define its \textit{Fatou set} $\F=\F((P_j))$ by
\[ \F = \{z \in \CC : \{P_n \circ \cdots \circ P_2 \circ P_1 \}_{n=1}^\infty
\mbox{ is a normal family on some neighborhood of } z \} \]
where we take our neighborhoods with respect to the spherical topology on $\CC$.  We then define the \textit{Julia set} $\J=\J((P_j))$ to be the complement $\CC \setminus \F$.

\begin{theorem}\label{ThmDichotomy}
Let $f$ be a polynomial on $\C$ of degree at least 2.  Suppose $f$ has no critical values in the closed unit disk $\overline{\D}$ and that $f^{-1}(\overline{\D}) \subset \D$.  Fixing a sequence $a_j \in \C$ with each $|a_j|>1$, we define polynomials $P_j(z)=a_jf(z)$.  Then
\begin{enumerate}[(1)]
\item $\J((P_j))$ is uniformly perfect if and only if $\limsup |a_j| < \infty$, and
\item $\J((P_j))$ is pointwise thin (and HNUP) if and only if $\limsup |a_j| = \infty$.
\end{enumerate}
\end{theorem}

\begin{remark}
For $a, c \in \C$ with $|c|>1$ and $|a|-|c|>1$, one may choose $f(z)=az^2+c$ in the above theorem.  Note then that $|z|\geq 1$ implies $|f(z)|=|az^2+c| \geq |a|-|c|>1$, i.e., $f(\C \setminus \D) \subseteq \C \setminus \overline{\D}$, which gives that $f^{-1}(\overline{\D}) \subset \D$.  Also, clearly the sole critical value of $f$ is $c \notin \overline{\D}$.  Hence applying the above theorem with such an $f$ and a suitable sequence $(a_j)$ with $\limsup |a_j| = \infty$, we can create a simple sequence of polynomials with pointwise thin (and thus HNUP) Julia set without the complicated arguments presented in~\cite{CSS}.
\end{remark}

\begin{proof}
(1) The Julia set of a bounded sequence of polynomials is known to be uniformly perfect (see Theorem~1.21 of~\cite{Su4}).

(2) Suppose $\limsup |a_j| = \infty$, and choose a subsequence $a_{j_n}$ such that $|a_{j_n}|\to \infty$.  We complete the proof by showing $\J((P_j))$ is pointwise thin and compact.  Calling $d$ the degree of $f$, we note that $f$ has $d$ well defined inverse branches $f_1, \dots, f_d$, on some open connected set $U=\Delta(0, 1+\epsilon) \supset \overline{\D}$ since all critical values of $f$ lie outside of $\overline{\D}$.  Furthermore, we note that we may choose $U$ such that $f^{-1}(U) \subset \overline{\D}$.  Hence, each $P_j$ has $d$ well defined inverse branches on $U$ given by $\varphi_i^{(j)}(z)=f_i(\frac{z}{a_j})$ for $i=1, \dots,d$.

For each $j \in \N$, let $\Phi^{(j)}=\{\varphi_1^{(j)}, \dots, \varphi_d^{(j)} \}$ and note that these families form an NIFS $\Phi$ on $(U, X)$ where $X=\overline{\D}$.  For each $j$, note that $\varphi^{(j)}_{i}(X)=f_{i}(\overline{\Delta}(0,\frac1{|a_{j}|})) \subset f_{i}(X)\subset \mathrm{Int}(X)$ for $i=1, \dots,d$.  Hence, $\Phi$ satisfies the Strong Separation Condition and, using the notation of Corollary~\ref{CorHNUP}, we also see that for each $n \in \N$,
$$b_{j_n}\geq b_0 := \min\{\mathrm{dist}(f_{i}(X),\partial X):i \in  \{1, \dots, d\} \}>0,$$
$$\delta_{j_n}\geq \delta_0 := \min\{\mathrm{dist}(f_{a}(X),f_{b}(X)):a,b \in \{1, \dots, d\} \textrm{ with }a \neq b \}>0$$  and
$$\eta_{j_n}=\max \{\mathrm{diam}(\varphi_i^{(j_n)}(X)):i \in  \{1, \dots, d\} \}$$
$$=\max \{\mathrm{diam}(f_{i}(\overline{\Delta}(0,\frac1{|a_{j_n}|}))):i \in  \{1, \dots, d\} \} \to 0.$$
Since $\inf\{b_{j_n}\}>0$, Corollary~\ref{CorHNUP} yields that $J(\Phi)$ is pointwise thin since $\frac{\delta_{j_n}}{\eta_{j_n}} \geq \frac{\delta_0}{\eta_{j_n}} \to \infty$.  Further, we note that $J(\Phi)$ is compact since each $I^{(j)}$ is finite.

The result then follows by showing that $\J((P_j))=J(\Phi)$.
Note that
$J(\Phi )=\{ z\in \Bbb{C}: P_{j}\circ \cdots \circ P_1(z)\in \overline{\Bbb{D}}
\mbox{ for each } j\}$.    Also note that $\C \setminus \overline{\D}$ is forward invariant under each $P_j$, and so it follows from Montel's Theorem that $\C \setminus J(\Phi) \subseteq \F((P_j))$, i.e., ${\mathcal J}((P_j)) \subseteq J(\Phi)$.
Since $J(\Phi)$ is pointwise thin, it is clear that $J(\Phi)$ has no interior.  This implies that any $z \in J(\Phi)$, which necessarily has as its orbit contained in the compact subset $f_1(X)\cup \dots \cup f_d(X)$ of $\D$, must be arbitrarily close to points whose orbits escape $\overline{\D}$.  Hence, $J(\Phi) \subseteq \J((P_j))$.
\end{proof}

\begin{corollary}\label{CorAlmostAll}
Let $f$ be a polynomial on $\C$ of degree at least 2.  Suppose $f$ has no critical values in $\overline{\D}$ and that $f^{-1}(\overline{\D}) \subset \D$.  Let $\tau$ be a probability measure on $\C \setminus \overline{\D}$ with unbounded support.  Then for almost all sequences $(a_j) \in \prod_{j=1}^\infty (\C \setminus \overline{\D})$ with respect to $\tilde{\tau}= \bigotimes_{j=1}^\infty \tau$, the maps $P_j=a_j \cdot f$ define a sequence of polynomials whose Julia set $\J((P_j))$ is pointwise thin.
\end{corollary}

\begin{proof}
For $N \in \N$, set $B_N=\{(a_j):|a_j| \leq N \textrm{ for all } j\}$ and note that since $\tau$ has unbounded support, $\tilde{\tau}(B_N)=0$ by the law of large numbers.  Hence, $\tilde{\tau}(\cup_{N \in \N} B_N)=0$, i.e., the set of bounded sequences has $\tilde{\tau}$-measure zero.  The result then follows from Theorem~\ref{ThmDichotomy}.
\end{proof}

\section{Proof of the Main Theorems}\label{SectProofs}

In this section we first prove Theorems~\ref{Thm2Main} and~\ref{ThmAnalNIFS} regarding uniform perfectness, and then prove Theorem~\ref{Thm2MainHNUP}, Corollary~\ref{CorHNUP}, and Theorem~\ref{ThmAnalNIFS-HNUP} regarding pointwise thinness.

We begin by proving a crucial lemma that will be key in providing a uniform Lipschitz constant for certain locally defined inverse maps.

\begin{lemma}\label{LemInverseLip}
Let $\mathcal{F}$ be a collection of analytic functions mapping non-empty open set $U \subset \C$ into compact set $X \subset U$ such that there exists $\eta>0$ where for all $f \in \mathcal{F}$ we have $|f'|\geq \eta$ on $X$.  Then there exists $r_0>0$ such that for every $f \in \mathcal{F}$ and $x \in X$, we have $|g'| \leq \frac{2}{\eta}$ on $\Delta(f(x), r_0)$ where $g$ is the local branch of the inverse of $f$ such that $g(f(x))=x$.
\end{lemma}

Note that this lemma does not require the maps $f \in \mathcal{F}$ to be M\"obius, or even globally conformal on $U$.

\begin{proof}
First note that by compactness, there exists $r>0$ such that for all $x \in X$ we have $\Delta(x,r) \subseteq U$.
Applying Lemma 2.3 of~\cite{RS4}, where $M>0$ is taken large enough so that $X \subset \Delta(0,M)$, we see that for some $\rho>0$ each $f \in \mathcal{F}$ is one-to-one on $\Delta(x, \rho)$ for every $x \in X$.  (Note that $\rho$ is independent of $f \in \mathcal{F}$ and $x \in X$.)
By the Koebe distortion theorem (see, e.g., Theorem 1.6 of~\cite{CG}), there exists $0<r_1 < \rho$ such that for every $f \in \mathcal{F}$ and $x \in X$, we have $|f'| \geq \frac{\eta}2$ on $\Delta(x, r_1)$.  By the Koebe 1/4 Theorem, for each $x \in X$ we then see that $f(\Delta(x,r_1)) \supseteq \Delta(f(x), \frac{r_1 \eta}4)$.  Hence, calling $r_0=\frac{r_1 \eta}4$ we have that a branch $g$ of $f^{-1}$ is defined on $\Delta(f(x), r_0)$ such that $g(f(x))=x$ and has $|g'| \leq \frac2{\eta}$ there.
\end{proof}

\begin{remark}\label{RemTopRow}
Under the hypotheses of Theorem~\ref{Thm2Main}, the Derivative Condition along with the distortion theorems used in the proof of the above lemma yield that
$\inf\{\mathrm{diam}(\varphi(X)):\varphi \in \cup_{j \in \N} \Phi^{(j)}\}>0$.  To see this, choose $x_0 \in X$ and $r >0$ such that $\Delta(x_0, r) \subset X$ (note that $X$ must have interior since it contains the open sets $\varphi(U)$ for all $\varphi \in \cup_{j \in \N} \Phi^{(j)}$).  Fixing $r < \rho$ from the above proof, we see that by the Koebe 1/4 Theorem, $\varphi(X) \supseteq \varphi(\Delta(x_0, r)) \supseteq \Delta(\varphi(x_0),\frac{r \eta}{4})$ for all $\varphi \in \cup_{j \in \N} \Phi^{(j)}$, which justifies the claim.
\end{remark}

\begin{proof}[Proof of Theorem~\ref{Thm2Main}]
We begin by replacing $X$, if it is not connected, by the connected $\widetilde{X} \subset U$ as in Remark~\ref{RemConnX}, noting that the hypotheses are still met.  Indeed, the M\"obius and Two Point Separation Conditions are clearly still satisfied with respect to $\widetilde{X} \supset X$.  The Derivative Condition also still holds with respect to $\widetilde{X} \supset X$ though not as trivially.  We show this by contradiction.  Assume $\varphi_n'(z_n) \to 0$ as $n \to \infty$ where each $z_n \in \widetilde{X}$ and each $\varphi_n \in \cup_{j \in \N} \Phi^{(j)}$.  By compactness we may suppose $z_n \to z_0 \in \widetilde{X}$.  Since by Montel's Theorem, the family $\cup_{j \in \N} \Phi^{(j)}$ is normal on $U$, we may suppose $\varphi_n$ converges normally on $U$ to some map $\varphi$.  Hence, we must have $\varphi'(z_0)=0$.  Since each map in $\cup_{j \in \N} \Phi^{(j)}$ is M\"obius, and thus one-to-one on $U$, we see by Hurwitz's Theorem that $\varphi$ must be constant.  This implies that for any $x \in X$, we must have $\varphi_n'(x) \to \varphi'(x)=0$, but this contradicts the Derivative Condition on $X$ which gives that each $|\varphi_n'(x)|\geq \eta$.

It suffices to prove $\overline{J^{(1)}}$ is uniformly perfect since clearly each sub-NIFS of $\Phi$ which generates  ${J^{(j)}}$ must also satisfy conditions (i)-(iii).    First note that by Remark~\ref{RemSep} we see that $\textrm{diam}(J^{(1)}) \geq \delta$ and so $J=J^{(1)}$ has more than one point.
Recalling Remark~\ref{RemTrueAnn} and Remark~\ref{RmkAnnInf}, we consider a \textit{true} annulus $A_1$ which separates $\overline{J}$ and which has modulus large enough so that any conformal annulus $B \subset \CC$ with $\mathrm{mod~}B \geq \mathrm{mod~}A_1 -1$ contains an essential true annulus $B' \subset B$ such that $\mathrm{mod~}B'=\frac13 \mathrm{mod~}B$.  Since the true annulus $A_1$ must also separate $J = \cap_{k=1}^\infty X_k$, we apply Lemma~\ref{LemmaAnn2} to obtain a true annulus $A \subset A_1$ which separates some $X_{k_0}$ and has $\mathrm{mod~}A = \mathrm{mod~}A_1 -1$.  We complete the proof by showing that there exists an upper bound on $\mathrm{mod~} A$.

Recall the superscript notation of Section~\ref{Intro}, in particular, that $X_{k_0}^{(1)} = X_{k_0}$.
By the invariance condition~(\ref{EqInvarCond}) in Remark~\ref{RmkInvariance}, we have $\bigcup_{i \in I^{(1)}} \varphi_i^{(1)}(X_{k_0-1}^{(2)})=X_{k_0}^{(1)},$ and so there must be some $\varphi_{i_1}^{(1)} \in \Phi^{(1)}$ such that $A$ surrounds some point of $\varphi_{i_1}^{(1)}(X_{k_0-1}^{(2)})$ (i.e., the bounded component of $\C \setminus A$ contains a point of $\varphi_{i_1}^{(1)}(X_{k_0-1}^{(2)})$).  Since $A$ separates $X_{k_0}^{(1)}$, we must have one of two cases: Case (I) $A$ surrounds all of $\varphi_{i_1}^{(1)}(X_{k_0-1}^{(2)})$, or Case (II) $A$ separates $\varphi_{i_1}^{(1)}(X_{k_0-1}^{(2)})$. See Figure~\ref{TableWithAnn}.

\begin{figure}[H]
  \centering
  \includegraphics[width=5.3in]{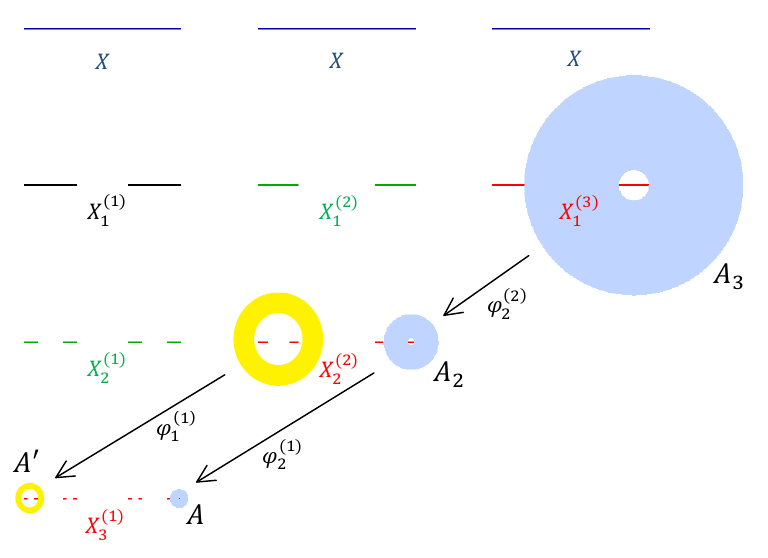}
  \caption{Illustration of the proof of Theorem~\ref{Thm2Main} using the system of Example~\ref{ExCantor}. Note that $A$ and $A_2=(\varphi_2^{(1)})^{-1}(A)$ are both of Case (II) type, whereas $A'$ is of Case~(II) type, but $(\varphi_1^{(1)})^{-1}(A')$ is of Case (I) type.}\label{TableWithAnn}
\end{figure}

\textbf{Case (I):} Write $A= \mathrm{Ann}(z;r,R)$ and suppose it surrounds all of $\varphi_{i_1}^{(1)}(X_{k_0-1}^{(2)})$.
Hence $(\varphi_{i_1}^{(1)})^{-1}(\overline{\Delta(z,r)}) \supseteq X_{k_0-1}^{(2)} \supseteq J^{(2)}$, which by the Two Point Separation Condition (see Remark~\ref{RemSep}) gives that $\mathrm{diam}\left((\varphi_{i_1}^{(1)})^{-1}(\overline{\Delta(z,r)})\right) \geq \delta$.

From Lemma~\ref{LemRbound} it follows that we only need to consider cases where $R < 2\cdot\mathrm{diam}(X)$.
We now establish an upper bound for $\mathrm{mod~}A = \log (R/r)$ by finding a positive lower bound for $r$.

Notice that due to the Derivative Condition and Lemma~\ref{LemInverseLip}, there exists $r_0>0$ such that for any $x \in X$ and $\varphi \in \cup_{j \in \N} \Phi^{(j)}$, we have $|(\varphi^{-1})'| \leq \frac{2}{\eta}$ on $\Delta(\varphi(x), r_0)$.

We now suppose $r<\min\{\frac{\delta \eta}{4},\frac{r_0}2\}$, from which we derive a contradiction, thus producing a lower bound for $r$ and completing the proof for Case (I).  Since $\overline{\Delta(z,r)}$ meets $\varphi_{i_1}^{(1)}(X_{k_0-1}^{(2)})$, we may choose $x_0 \in X_{k_0-1}^{(2)} \subseteq X$ such that $\varphi_{i_1}^{(1)}(x_0) \in \overline{\Delta(z,r)} \subset \overline{\Delta(\varphi_{i_1}^{(1)}(x_0), 2r)}\subset \Delta(\varphi_{i_1}^{(1)}(x_0), r_0)$.  Since $\left|\left((\varphi_{i_1}^{(1)})^{-1}\right)'\right|\leq \frac{2}{\eta}$ on $\Delta(\varphi_{i_1}^{(1)}(x_0), r_0)$ which contains the convex set $\overline{\Delta(z,r)}$, we see that $\mathrm{diam}\left((\varphi_{i_1}^{(1)})^{-1}(\overline{\Delta(z,r)})\right) \leq \frac{4r}{\eta} < \delta$, which is a contradiction.

\textbf{Case (II):}  Suppose $A$ separates $\varphi_{i_1}^{(1)}(X_{k_0-1}^{(2)})$.  Hence, the conformal annulus  $A_2=(\varphi_{i_1}^{(1)})^{-1}(A)$ must separate $X_{k_0-1}^{(2)}$ and must have $\mathrm{mod~}A_2=\mathrm{mod~}A$.
In terms of Figure~\ref{TableWithAnn}, we have constructed an annulus $A_2$ which separates $X_{k_0-1}^{(2)}$ in the picture diagonally up and right of the picture of $X_{k_0}^{(1)}$.  
Note, however, that $A_2$ will contain $\infty$ when $\varphi_{i_1}^{(1)}(\infty) \in A$, and so we must allow for this possibility.

Hence we may repeat our process as follows.  Since $A_2$ separates the set $X_{k_0-1}^{(2)}=\bigcup_{i \in I^{(2)}} \varphi_i^{(2)}(X_{k_0-2}^{(3)})$, we must have at least one of two cases: Case (I') one component of
$\CC \setminus A_2$ contains some $\varphi_{i_2}^{(2)}(X_{k_0-2}^{(3)})$ while the other component contains some other  $\varphi_{i_2'}^{(2)}(X_{k_0-2}^{(3)})$, or Case (II') $A_2$ separates some $\varphi_{i_2}^{(2)}(X_{k_0-2}^{(3)})$.  If Case (I') holds, extract an essential true annulus $A_2' \subset A_2$ with $\mathrm{mod~}A_2'= \frac13 \mathrm{mod~}A_2= \frac13 \mathrm{mod~}A$, which must surround all of either $\varphi_{i_2}^{(2)}(X_{k_0-2}^{(3)})$ or $\varphi_{i_2'}^{(2)}(X_{k_0-2}^{(3)})$,  and then bound $\mathrm{mod~}A'_2$ as in Case (I) above.  If Case (II') holds, we repeat the process of Case (II) above, noting that we do not need to first extract a true annulus from $A_2$.



This process must then end by eventually applying the method of Case (I), or by eventually producing (after $k_0$ steps) an annulus $A_{k_0}$, with the same modulus as of $A$, which separates $X_1^{(k_0)}$.  The proof is thus concluded by showing that such a modulus is uniformly bounded independent of the choice of $k_0$.  First, extract an essential true annulus $A_{k_0}'=\mathrm{Ann}(z';r',R') \subset A_{k_0}$ with $\mathrm{mod~}A_{k_0}'= \frac13 \mathrm{mod~}A_{k_0}= \frac13 \mathrm{mod~}A$, which necessarily separates $X_1^{(k_0)}$.  Again by Lemma~\ref{LemRbound}, it is then clear that we only need to produce a lower bound for $r'$.  This follows easily from Remark~\ref{RemTopRow} by noting that $\overline{\Delta(z',r')}$ would need to contain the connected set $\varphi(X)$ for some $\varphi \in \Phi^{(k_0)}$.

Examination of the above proof shows that $\mathrm{mod~}A_1$ is bounded above by a constant which depends only on $\delta$ and $\eta$.
\end{proof}

Note that the step of extracting a true annulus of one-third the modulus is done only at most once in the above proof.

\begin{proof}[Proof of Theorem~\ref{ThmAnalNIFS}]
By Proposition~\ref{PropUPImage}, for each $\varphi \in \widetilde{\Phi}^{(1)}$ the set $\phi\left(\overline{J^{(n)}}\right)$ is uniformly perfect. Lemma~\ref{LemInvClosedLimitSets} gives that $\overline{J^{(1)}}=\bigcup_{\varphi \in \widetilde{\Phi}^{(1)}}\varphi\left(\overline{J^{(n)}}\right)$, and the result follows since the finite union of uniformly perfect sets is uniformly perfect.
\end{proof}

We now prove the Theorem~\ref{Thm2MainHNUP}, Corollary~\ref{CorHNUP}, and Theorem~\ref{ThmAnalNIFS-HNUP} regarding pointwise thinness.

\begin{proof}[Proof of Theorem~\ref{Thm2MainHNUP}]
Note that since the NIFS $\Phi$ is conformal and both the annulus $A_{j_n}$ and its bounded complementary component lie inside $X \subset U$, we see that $\pi_\Phi(\omega) \in \varphi_{\omega_1 \cdots \omega_{j_n-1}}(\varphi_{m_n}^{(j_n)}(X))$ (see Remark~\ref{RemPiecesContainLimitPts}) is surrounded by the conformal annulus $A_{j_n}'=\varphi_{\omega_1 \cdots \omega_{j_n-1}}(A_{j_n})$, which separates $\varphi_{\omega_1 \cdots \omega_{j_n-1}}(X_1^{(j_n)})\subseteq X_{j_n}^{(1)}$.  See Figure~\ref{FigImageAnn}.  We claim that $A_{j_n}' \cap X_{j_n}^{(1)}=\emptyset$, from which it follows that $A_{j_n}'$ separates $X_{j_n}^{(1)}$, and thus separates $J$.  Since $\mathrm{mod}~A_{j_n'}=\mathrm{mod}~A_{j_n} \to \infty$ with $\mathrm{diam}(A_{j_n}') \to 0$, we see that $J$ is pointwise thin at $\pi_\Phi(\omega)$.

To prove the claim, suppose towards a contradiction that $A_{j_n}'$ meets
\begin{eqnarray*}
X_{j_n}^{(1)}=\bigcup_{\omega^* \in I^{j_n}} \varphi_{\omega^*}(X)
&=&\bigcup_{\omega^*_1 \cdots \omega^*_{j_n-1} \in I^{j_n-1}} \bigcup_{\omega^*_{j_n} \in I^{(j_n)}} \varphi_{\omega^*_1 \cdots \omega^*_{j_n-1}}(\varphi_{\omega^*_{j_n}}(X))\\
&=&\bigcup_{\omega^*_1 \cdots \omega^*_{j_n-1} \in I^{j_n-1}}  \varphi_{\omega^*_1 \cdots \omega^*_{j_n-1}}(X_1^{(j_n)}).
\end{eqnarray*}
Hence, $A_{j_n}'$  meets
$\varphi_{\omega^*_1 \cdots \omega^*_{j_n-1}}(X_1^{(j_n)})$ for some $\omega^*_1 \cdots \omega^*_{j_n-1} \in I^{j_n-1}$.
Note that $\omega^*_1 \cdots \omega^*_{j_n-1} \neq \omega_1 \cdots \omega_{j_n-1}$ since $A_{j_n}'$ separates $\varphi_{\omega_1 \cdots \omega_{j_n-1}}(X_1^{(j_n)})$.  However, since $X_{1}^{(j_n)} \subseteq X$, $A_{j_n} \subseteq X$, and $\varphi_{\omega^*_1 \cdots \omega^*_{j_n-1}}(X) \cap \varphi_{\omega_1 \cdots \omega_{j_n-1}}(X) = \emptyset$ by the strong separation condition (see the discussion preceding Definition~\ref{Wordsdef}), we see that
$\varphi_{\omega^*_1 \cdots \omega^*_{j_n-1}}(X_{1}^{(j_n)})$ cannot meet $A_{j_n}'=\varphi_{\omega_1 \cdots \omega_{j_n-1}}(A_{j_n})$, which is a contradiction.
\end{proof}

\begin{figure}
  \centering
  \includegraphics[width=5.6in]{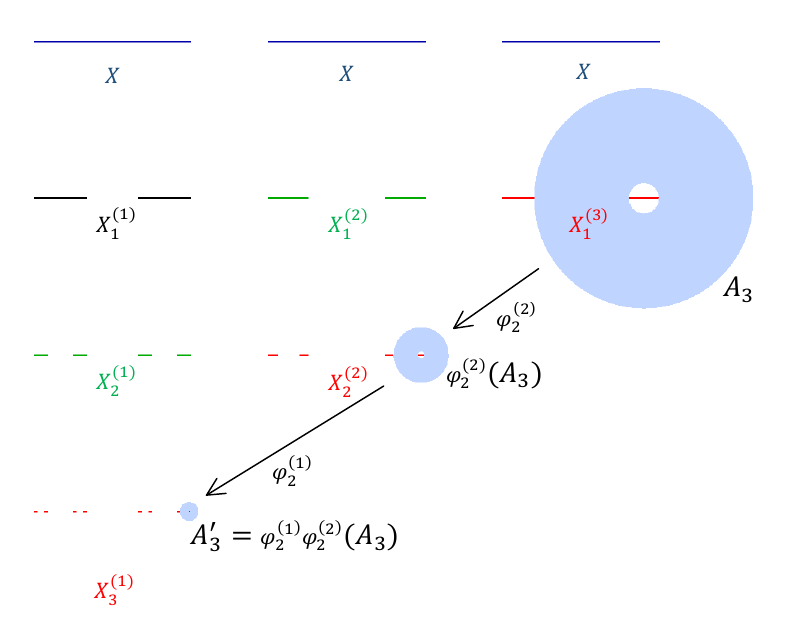}
  \caption{Table illustrating the proof of Theorem~\ref{Thm2Main} using the system of Example~\ref{ExCantor}. }\label{FigImageAnn}
\end{figure}

\begin{proof}[Proof of Corollary~\ref{CorHNUP}]
Pick an arbitrary $\omega \in I^\infty$.  For each $n$, choose some $z_n\in \varphi_{\omega_{j_n}}^{(j_n)}(X)$, and define $A_{{j_n}}=\mathrm{Ann}(z_n;\eta_{j_n},\frac{\delta_{j_n}}{c})$, which by definition of $\eta_{j_n}$ must surround $\varphi_{\omega_{j_n}}^{(j_n)}(X)$.  Hence by definition of $\delta_{j_n}$, the annulus $A_{{j_n}}$ must separate $X^{(j_n)}_1$.  Lastly, since $\frac{\delta_{j_n}}{c} \leq b_{j_n} \leq \mathrm{dist}(\varphi^{(j_n)}_{\omega_{j_n}}(X),\partial X)$, we see that $\Delta(z_n,\frac{\delta_{j_n}}{c}) \subseteq X$.  Thus by Theorem~\ref{Thm2Main}, noting that $\mathrm{mod~}A_{j_n} =\log \frac{\delta_{j_n}}{c\eta_{j_n}} \to \infty$, we see that $J$ is pointwise thin at $\pi_\Phi(\omega)$.  The proof is then complete by noting $J=\pi_\Phi(I^\infty)$ since $\Phi$ satisfies the Strong Separation Condition (as mentioned just before the statement of Lemma~\ref{LemProjImageClosure}).
\end{proof}

\begin{proof}[Proof of Theorem~\ref{ThmAnalNIFS-HNUP}]
Consider the analytic NIFS $\widetilde{\Phi}$ given by $\widetilde{\Phi}^{(1)}=\Phi^{(1)} \circ \cdots \circ \Phi^{(n-1)}$ and $\widetilde{\Phi}^{(j)}=\Phi^{(j+n-2)}$ for each $j>1$.  Hence, by Remark~\ref{RemCombStages}, we see that $J(\Phi)=J(\widetilde{\Phi})$.
By Proposition~\ref{PropPointThinImage}, for each $\varphi \in \widetilde{\Phi}^{(1)}$ the set $\phi\left(\overline{J^{(n)}}\right)$ is pointwise thin. Lemma~\ref{LemInvClosedLimitSets} gives that $\overline{J^{(1)}}=\bigcup_{\varphi \in \widetilde{\Phi}^{(1)}}\varphi\left(\overline{J^{(n)}}\right)$, and the result follows since the finite disjoint union of compact pointwise thin sets is pointwise thin.
\end{proof}

\section*{Acknowledgments} This work was partially supported by a grant from the Simons Foundation (\#318239 to Rich Stankewitz).  The last author was partially supported by
JSPS Kakenhi 19H01790.

\end{document}